\documentclass{amsart}
\usepackage{amsfonts,amsmath,amssymb,times,hyperref,tikz,subfigure}

\theoremstyle{plain}
\newtheorem{thm}{Theorem}[section]

\newtheorem{cor}[thm]{Corollary}

\newtheorem{prop}[thm]{Proposition}
\newtheorem{lemma}[thm]{Lemma}

\newcounter{example}
\newenvironment{example}[1][]{\refstepcounter{example}\par\medskip
\noindent \textbf{Example~\theexample. #1} \rmfamily}{\medskip}

\numberwithin{equation}{section}

\DeclareMathOperator{\cl}{cl}

\newcommand{\del}{\setminus}
\newtheorem{sublemma}{}[thm]
\newtheorem{subsublemma}{}[sublemma]

\theoremstyle{definition}
\newtheorem*{ex:1}{Example \ref{example:1} (continued)}
\newtheorem*{ex:2}{Example \ref{ex:k-natural} (continued)} 

\hypersetup{colorlinks=true, linkcolor=blue, urlcolor=blue,citecolor=blue}

\title[Natural matroids and excluded minors]{The excluded minors for
  three classes of $2$-polymatroids having special types of natural
  matroids}
\author[J.~Bonin]{Joseph E.~Bonin} \address
{Department of Mathematics\\ The George Washington University\\
  Washington, D.C.\ 20052, USA} \email {jbonin@gwu.edu,
  kevinlong@gwmail.gwu.edu} \author[K.~Long]{Kevin Long}
\date{\today}

\begin{document}

\begin{abstract}
  If $\mathcal{C}$ is a minor-closed class of matroids, the class
  $\mathcal{C}'$ of integer polymatroids whose natural matroids are in
  $\mathcal{C}$ is also minor closed, as is the class $\mathcal{C}'_k$
  of $k$-polymatroids in $\mathcal{C}'$.  We find the excluded minors
  for $\mathcal{C}'_2$ when $\mathcal{C}$ is (i) the class of binary
  matroids, (ii) the class of matroids with no $M(K_4)$-minor, and,
  combining those, (iii) the class of matroids whose connected
  components are cycle matroids of series-parallel networks.  In each
  case the class $\mathcal{C}$ has finitely many excluded minors, but
  that is true of $\mathcal{C}'_2$ only in case (ii).  We also
  introduce the $k$-natural matroid, a variant of the natural matroid
  for a $k$-polymatroid, and use it to prove that these classes of
  $2$-polymatroids are closed under $2$-duality.
\end{abstract}

\maketitle

\section{Introduction}\label{section:introduction}

An \emph{integer polymatroid} is a pair $(E, \rho)$ where $E$ is a
finite set and $\rho:2^E\to\mathbb{Z}$ is a function that satisfies
the following properties:
\begin{enumerate}
\item $\rho$ is \emph{normalized}, that is, $\rho(\emptyset)=0$,
\item $\rho$ is \emph{nondecreasing}, that is, if
  $A\subseteq B\subseteq E$, then $\rho(A)\leq \rho(B)$, and
\item $\rho$ is \emph{submodular}, that is,
  $\rho(A\cup B)+\rho(A\cap B)\leq \rho(A)+\rho(B)$ for all
  $A, B\subseteq E$.
\end{enumerate}
We focus exclusively on integer polymatroids, so we refer to them as
\emph{polymatroids}.  We often refer to a polymatroid by its
\emph{rank function} $\rho$, calling $\rho$ a polymatroid on $E$.  A
polymatroid $(E, \rho)$ is a \emph{$k$-polymatroid} if $\rho(i)\leq k$
for all $i\in E$.  Thus, matroids are $1$-polymatroids.  Polymatroids
generalize matroids by allowing elements of higher rank: while
matroids consist of loops (elements of rank $0$) and points (elements
of rank $1$), polymatroids can contain loops, points, lines (elements
of rank $2$), planes (elements of rank $3$), and so on.  Via the rank
function, the matroid notions of deletion, contraction, and minors
generalize directly to polymatroids; see Section
\ref{section:background} for the definitions.
		
The class of matroids that are representable over a field $\mathbb{F}$
is minor closed, and so can be characterized by its excluded minors.
A well-known conjecture due to Rota states that if $\mathbb{F}$ is a
finite field, then the list of excluded minors for the class of
$\mathbb{F}$-representable matroids is finite.  Geelen, Gerards, and
Whittle \cite{Rota's conjecture} announced a proof of this conjecture.
	
The situation for polymatroids is more difficult.  Oxley, Semple, and
Whittle \cite{wheels and whirls} constructed an infinite family of
excluded minors for $\mathbb{F}$-representable $2$-polymatroids, so a
direct counterpart of Rota's conjecture for polymatroids is not true.
They suggest a variant of Rota's conjecture for polymatroids,
conjecturing that $2$-polymatroids have finitely many excluded
p-minors (a variant of minors that has a third reduction operation in
addition to deletion and contraction).  The complete list of excluded
minors for binary $2$-polymatroids is not known.
	
Characterizing the class of binary $2$-polymatroids by excluded minors
appears to be a difficult problem.  We treat a simpler problem: we
find the excluded minors for the subclass of $2$-polymatroids whose
natural matroids are binary.  Geometrically, the natural matroid of a
polymatroid $(E, \rho)$ is the matroid obtained by replacing each
$e\in E$ with $\rho(e)$ points lying freely in $e$, as illustrated in
Figure \ref{fig:exceptions}. (For the precise definition of the natural
matroid, see Section \ref{section:natural}.)  Theorem \ref{thm: main
  theorem} gives the excluded minors for the class of $2$-polymatroids
whose natural matroids are binary.  There are infinitely many excluded
minors.  It is often difficult to find the excluded minors for a
minor-closed class of matroids that has infinitely many excluded
minors, such as base-orderable matroids \cite{non-base-orderable, sbo}
and gammoids \cite{gammoids, sbo}.  Classes of matroids for which the
excluded minors are known and there are infinitely many include nested
matroids \cite{nested}, lattice path matroids \cite{lpm}, laminar
matroids \cite{laminar}, $2$-laminar matroids and $2$-closure-laminar
matroids \cite{generalized laminar}; classes of polymatroids of this
type include Boolean polymatroids \cite{Matus}, the union, over all
$k\geq 1$, of the classes of $k$-quotient polymatroids
\cite{quotient}, and the class of lattice path polymatroids
\cite{LPP}.

The class of binary matroids has one excluded minor, $U_{2,4}$, and,
as just stated, the class of $2$-polymatroids whose natural matroids
are binary has infinitely many excluded minors.  In contrast, the
class of $2$-polymatroids whose natural matroids have no
$M(K_4)$-minor has only finitely many excluded minors; we find them in
Section \ref{section:K_4} (see Theorem \ref{thm:noK4exmin}).  Cycle
matroids of series-parallel networks are binary, have no
$M(K_4)$-minor, and are connected.  We are interested in the broader
class of what we call \emph{series-parallel matroids}, which are
direct sums of cycle matroids of series-parallel networks.  The class
of series-parallel matroids is minor closed; its excluded minors are
$U_{2,4}$ and $M(K_4)$.  With Theorems \ref{thm: main theorem} and
\ref{thm:noK4exmin}, we find the excluded minors for the class of
$2$-polymatroids whose natural matroids are series-parallel matroids
(see Theorem \ref{thm:spexmin}); there are infinitely many excluded
minors.

\begin{figure}[t]
\centering
\subfigure[$Z_3$]{
  \begin{tikzpicture}[scale=1]
  \draw[thick](0,1)--(1.5,2.25);
  \draw[thick](1.5,2.25)--(1.5,0.2);
  \draw[thick](1.5,2.25)--(3,1);
  \draw[thin, gray!50](1.5,0.2)--(0,1);
  \draw[thin, gray!50](1.55,1)--(3,1);
  \draw[thin, gray!50](1.5,0.2)--(3,1);
  \draw[thin, gray!50](1.45,1)--(0,1);
  \node at (.7,1.72) {\footnotesize $a$};
  \node at (1.6,1.5) {\footnotesize $b$};
  \node at (2.3,1.72) {\footnotesize $c$};
  \label{subfig:Z_3}
  \end{tikzpicture}}
\hspace{20pt}
\subfigure[$Z_{2,2}$]{
  \begin{tikzpicture}[scale=1]
  \draw[thick](0,1)--(1.5,2.25);
  \draw[thin, gray!50](1.5,2.25)--(1.5,0.2);
  \draw[thick](1.5,2.25)--(3,1);
  \draw[thin, gray!50](1.5,0.2)--(0,1);
  \draw[thin, gray!50](1.55,1)--(3,1);
  \draw[thin, gray!50](1.5,0.2)--(3,1);
  \draw[thin, gray!50](1.45,1)--(0,1);
  \node at (.7,1.72) {\footnotesize $a$};
  \filldraw (1.5,1.4) node[right] {\footnotesize $b_2$} circle (2pt);
  \filldraw (1.5,0.6) node[right] {\footnotesize $b_1$} circle (2pt);	
  \node at (2.3,1.72) {\footnotesize $c$};
  \label{subfig:Z_2,2}
  \end{tikzpicture}}
\hspace{20pt}
\subfigure[$M_{Z_3}=M_{Z_{2,2}}$]{
  \begin{tikzpicture}[scale=1]
  \draw[thin, gray!50](0,1)--(1.5,2.25);
  \draw[thin, gray!50](1.5,2.25)--(1.5,0.2);
  \draw[thin, gray!50](1.5,2.25)--(3,1);
  \draw[thin, gray!50](1.5,0.2)--(0,1);
  \draw[thin, gray!50](1.55,1)--(3,1);
  \draw[thin, gray!50](1.5,0.2)--(3,1);
  \draw[thin, gray!50](1.45,1)--(0,1);		
  \filldraw (0,1) node[above] {\footnotesize $a_1$} circle  (2pt);
  \filldraw (0.75,1.625) node[above] {\footnotesize $a_2$} circle  (2pt);
  \filldraw (1.5,0.6) node[right] {\footnotesize $b_1$} circle (2pt);	
  \filldraw (1.5,1.4) node[right] {\footnotesize $b_2$} circle (2pt);
  \filldraw (3,1)  node[above] {\footnotesize $c_1$} circle  (2pt);
  \filldraw (2.25,1.625)  node[above] {\footnotesize $c_2$} circle  (2pt);
  \label{subfig:natural}
  \end{tikzpicture}}
\caption{Parts \subref{subfig:Z_3} and \subref{subfig:Z_2,2} each show
  a rank-$4$ polymatroid.  Part \subref{subfig:natural} shows the
  natural matroid for both polymatroids.}
\label{fig:exceptions}
\end{figure}
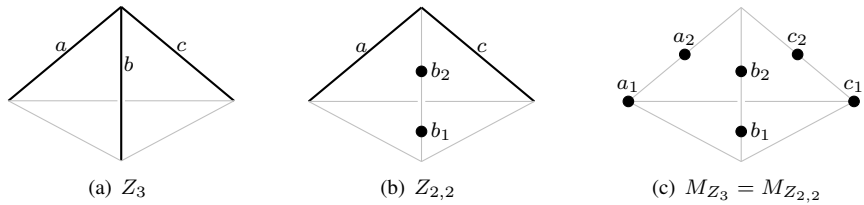

Sections \ref{section:background} through \ref{section:compression}
pave the way for that work.  Background on polymatroids and the
natural matroid of a polymatroid appears in Section
\ref{section:background}.  In Section \ref{sec:classes}, we discuss
the classes of polymatroids that we focus on.  In Section
\ref{section:k-natural}, we introduce the $k$-natural matroid of a
polymatroid, a variation of the natural matroid.  With this
construction, which is more compatible with $k$-duality than is the
natural matroid, we prove that the classes of $2$-polymatroids that we
treat are closed under $2$-duality.  In Section
\ref{section:compression}, we analyze how the polymatroid operation of
compression interacts with taking the natural matroid, the $k$-natural
matroid, and the $2$-dual.  A key result there is Lemma \ref{lemma:
  compression of excluded minor}, about excluded minors that have
compressions that are also excluded minors, which we use to reduce the
problem of finding excluded minors to two smaller problems.
Compression reduces the only infinite sequence of excluded minors in
Theorems \ref{thm: main theorem} and \ref{thm:spexmin} to a single
excluded minor.

\section{Background}\label{section:background}

Our terminology and notation for matroid theory follow Oxley
\cite{oxley}.  We let $[n]$ be the set $\{1,2,\ldots,n\}$.
As $E(M)$ denotes the ground set of a matroid $M$, so $E(\rho)$ is the
ground set of a polymatroid $\rho$.

\subsection{Polymatroids}
Given a polymatroid $(E, \rho)$ and some $X\subseteq E$, the
\emph{deletion of} $X$ is the polymatroid $(E-X,\rho_{\setminus X})$
where $\rho_{\setminus X}(Y)=\rho(Y)$ for $Y\subseteq E-X$.  The
deletion $\rho_{\setminus X}$ is also called the \emph{restriction of}
$\rho$ to $E-X$, denoted $\rho_{|E- X}$.  The \emph{contraction of}
$X$ is the polymatroid $(E- X,\rho_{/X})$ where
$\rho_{/X}(Y)= \rho(Y\cup X)-\rho(X)$ for $Y\subseteq E- X$.  When
$X=\{e\}$, we may write $\rho_{\setminus X}$ as $\rho_{\setminus e}$
and $\rho_{/X}$ as $\rho_{/e}$.  As for matroids, if $X$ and $Y$ are
disjoint, then $(\rho_{\setminus X})_{/Y}=(\rho_{/Y})_{\setminus X}$,
$(\rho_{\setminus X})_{\setminus Y}=\rho_{\setminus X\cup Y}$, and
$(\rho_{/X})_{/Y}=\rho_{/X\cup Y}$.  Sequences of these operations
produce minors; equivalently, a polymatroid $\rho'$ is a \emph{minor}
of $\rho$ if $\rho'= (\rho_{\setminus X})_{/Y}$ for some disjoint
subsets $X$ and $Y$ of $E$.

A class $\mathcal{C}$ of polymatroids is \emph{minor closed} if,
whenever $\rho\in \mathcal{C}$, all minors of $\rho$ are in
$\mathcal{C}$.  A minor-closed class $\mathcal{C}$ of polymatroids is
characterized by its \emph{excluded minors}, that is, the polymatroids
that are not in $\mathcal{C}$ but all of their proper minors are in
$\mathcal{C}$.

For $X\subseteq E$, its \emph{closure} is
$\cl(X)=\{e\in E\,:\,\rho(X\cup e)=\rho(X)\}$.  Like matroid closure,
polymatroid closure has the properties that $X\subseteq \cl(X)$ and
$\cl(\cl(X))=\cl(X)$ for all $X\subseteq E$, and if
$X\subseteq Y\subseteq E$, then $\cl(X)\subseteq \cl(Y)$.  A set $X$
\emph{spans} an element $e$ if $e\in\cl(X)$, and it \emph{spans} a set
$Y$ if $Y\subseteq \cl(X)$.  A set $F$ is a \emph{flat} if $F=\cl(X)$
for some $X\subseteq E$.

Elements $e, f\in E$ are \emph{parallel} if
$0<\rho(e)=\rho(f)=\rho(\{e, f\})$.  A $2$-polymatroid can have
parallel points and parallel lines.  For $X\subseteq E$, the elements
of $X$ are \emph{collinear} if $\rho(X)=2$, and \emph{coplanar} if
$\rho(X)=3$.  Elements $e$ and $f$ are \emph{skew} if neither is a
loop and $\rho(e)+\rho(f)=\rho(\{e,f\})$.  If $0<\rho(e)<\rho(f)$ and
$\rho(f)=\rho(\{e,f\})$, then $e$ \emph{lies on} $f$.
	
For polymatroids $(E_1, \rho_1)$ and $(E_2, \rho_2)$ with disjoint
ground sets, their \emph{direct sum} is the polymatroid
$(E_1\cup E_2,\rho_1\oplus \rho_2)$ where
$(\rho_1\oplus\rho_2)(X)=\rho_1(X\cap E_1)+\rho_2(X\cap E_2)$.  A
polymatroid is \emph{connected} if it is not the direct sum of two
nonempty polymatroids.

\subsection{The natural matroid}\label{section:natural}

Given a polymatroid $(E,\rho)$, its natural matroid $M_\rho$ is
defined as follows.  For each $e\in E$, let $X_e$ be a set of
$\rho(e)$ elements where the sets $X_e$, for all $e\in E$, are
pairwise disjoint.  For $A\subseteq E$, let
$$X_A=\bigcup_{e\in A}X_e.$$  The \emph{natural
  matroid} of $\rho$, denoted $M_\rho$, is the matroid $(X_E, r)$
where
\begin{equation}\label{eq:NatMtdRk}
  r(X)=\min\{\rho(A)+|X-X_A|\,:\, A\subseteq E\}.
\end{equation}
A result of McDiarmid \cite{natural proof} shows that this is a
matroid.  Note that if $\rho$ is a loopless matroid and we take
$X_e=\{e\}$ for each $e\in E$, then the natural matroid of $\rho$ is
$\rho$.  Also, if $e\in E$ and $e_1, e_2\in X_e$, then $e_1$ and $e_2$
are \emph{clones}, that is, the transposition $(e_1,e_2)$ that swaps
$e_1$ and $e_2$ and fixes all other elements of $X_E$ is an
automorphism of $M_\rho$.  We call $X\subseteq E$ a \emph{set of
  clones} if $a,b\in X$ are clones whenever $a\ne b$.

Helgason \cite{Helgason} introduced the natural matroid to prove the
following result \cite{Helgason, Lovasz, natural proof}.

\begin{thm}\label{thm:Helgason}
  A function $\rho:2^E\to \mathbb{Z}$ is a polymatroid if and only if
  there is a matroid $M$ on some set $E'$ and a function
  $\phi:E\to 2^{E'}$ with $\rho(A)=r_M(\bigcup_{e\in A} \phi(e))$ for
  all $A\subseteq E$.
\end{thm}

We will use the lemma below, which is a weaker version of a result in
\cite{natural matroid}; the original result further limits the sets
$X_A$ that need to be considered.

\begin{lemma}\label{lemma:shownatural}
  Let $\rho$, $E$, $X_e$, and $X_A$ be as above.  A matroid $M$ on
  $X_E$ is $M_\rho$ if and only if each set $X_e$ is a set of clones
  and $r_M(X_A)=\rho(A)$ for all $A\subseteq E$.
\end{lemma}

Observe that if $\rho$ and $\rho'$ are polymatroids on $E$ with
$M_\rho=M_{\rho'}$ and, for each $e\in E$, the corresponding set $X_e$
is the same in both natural matroids, then $\rho=\rho'$.

\begin{example}\label{example:1}
  Different polymatroids may have the same natural matroid, as Figure
  \ref{fig:exceptions} illustrates.  There, $Z_3$ is the rank-$4$
  polymatroid consisting of three pairwise coplanar lines, and
  $Z_{2,2}$ is the rank-$4$ polymatroid obtained by freely placing two
  points on a line $L$ in $Z_3$ (made precise below) and then deleting
  $L$.
\end{example}

\subsection{Principal extension}
Let $M=(E, r)$ be a matroid.  For $X\subseteq E$ and $e\notin E$,
define $r':2^{E\cup e}\to \mathbb{Z}$ by, for all $Y\subseteq E$,
setting $r'(Y)=r(Y)$ and
$$r'(Y\cup e)=\begin{cases}
  r(Y), & \text{if } X\subseteq\cl(Y)\\
  r(Y)+1, & \text{otherwise.}
\end{cases}$$ It is easy to check that $r'$ is the rank function of a
matroid on $E\cup e$.  This matroid, denoted $M+_X e$, is the
\emph{principal extension} of $M$ in which $e$ has been \emph{freely
  added} to $X$.  The \emph{free extension} of $M$ by $e$ is $M+_Ee$.
Note that $M+_X e = M+_{\cl(X)} e$, that $(M+_X e)\setminus e=M$, and
that $\cl(X)\cup e$ is a flat of $M+_X e$.  Also,
$(M+_X e_1)+_X e_2=(M+_X e_2)+_X e_1$, so iteratively adding
$e_1,e_2,\ldots,e_k$ to $X$ is independent of the order in which
$e_1,e_2,\ldots,e_k$ are added; the resulting matroid, denoted
$M+_X\{e_1, e_2,\ldots, e_k\}$, is an \emph{iterated principal
  extension} of $M$.

We use Theorem \ref{thm:Helgason} to define principal extensions of
polymatroids.  Fix a polymatroid $\rho$ on $E$, an element
$e\not\in E$, a nonnegative integer $k$, and a subset $F\subseteq E$
with $\rho(F)\geq k$.  Take the natural matroid $M_\rho=(X_E,r)$ of
$\rho$ and fix a $k$-element set $X_e$ with $X_e\cap X_E=\emptyset$.
For $A\subseteq E$, let $X_{A\cup e}=X_A\cup X_e$.  Let $(N,r')$ be
the iterated principal extension $M_\rho+_{X_F}X_e$ of $M_\rho$.
Define $\phi:E\cup e\to 2^{X_{E\cup e}}$ by $\phi(f)=X_f$ for
$f\in E\cup e$ and apply Theorem \ref{thm:Helgason}.  This yields a
polymatroid $(E\cup e, \rho')$ where $\rho'(A)=r'(X_A)$ for
$A\subseteq E\cup e$.  This polymatroid $\rho'$ is the \emph{principal
  extension} of $\rho$ in which the rank-$k$ element $e$ has been
\emph{freely added} to $F$.  If $F=E$, then $\rho'$ is the \emph{free
  extension} of $\rho$ by $e$.

\section{The classes of polymatroids of interest}\label{sec:classes}

The following observation is made in \cite{natural matroid}.

\begin{lemma}\label{lem:naturalminor}
  For any polymatroid $\rho$, we have
  $M_{\rho_{\setminus e}}=M_\rho\setminus X_e$ for all $e\in E$.
  Likewise, $M_{\rho_{/e}}=M_\rho/X_e|U$ where
  $|U\cap X_f|= \rho_{/e}(f)$ for each $f\in E-e$.
\end{lemma}

This simple lemma yields the observation that is the starting point
for this paper: if $\mathcal{C}$ is a minor-closed class of matroids
and $\mathcal{C}'$ is the class of polymatroids whose natural matroids
are in $\mathcal{C}$, then $\mathcal{C}'$ is minor closed.  Similarly,
the class $\mathcal{C}'_k$ of $k$-polymatroids in $\mathcal{C}'$ is
also minor closed since minors of $k$-polymatroids are
$k$-polymatroids.

We are interested in $\mathcal{C}'_2$ for three minor-closed classes
$\mathcal{C}$ of matroids: binary matroids; matroids with no
$M(K_4)$-minor; and series-parallel matroids.  In each case, we find
the excluded minors for $\mathcal{C}'_2$.

A polymatroid is $\emph{binary-natural}$ if its natural matroid is
binary.

\begin{ex:1}
  The matroid in part (c) of Figure \ref{fig:exceptions} is the
  parallel connection of three $3$-circuits at a common base point,
  with the base point deleted.  Parallel connections of binary
  matroids are binary, so $Z_3$ and $Z_{2,2}$ are binary-natural.
\end{ex:1}

Let $\mathcal{P}_{U_{2,4}}$ denote the class of binary-natural
$2$-polymatroids.  The uniform matroid $U_{2,4}$ is the unique
excluded minor for the class of binary matroids, so a $2$-polymatroid
$\rho$ is in $\mathcal{P}_{U_{2,4}}$ if and only if $U_{2,4}$ is not a
minor of $M_\rho$.  Also, $U_{2,4}$ is an excluded minor for
$\mathcal{P}_{U_{2,4}}$ and it is the only excluded minor that is a
matroid.  Any polymatroid (automatically a $2$-polymatroid) whose
natural matroid is $U_{2,4}$ is an excluded minor for
$\mathcal{P}_{U_{2,4}}$.  There are three such polymatroids:
$U_{2,4}$; a line with two points on it, denoted $L_2$; and two
parallel lines, denoted $S_2$. See Figure \ref{fig:natural}.

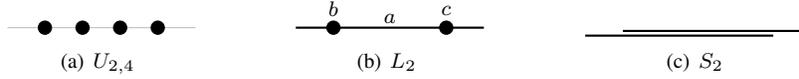
\begin{figure}[t]
\centering
\subfigure[$U_{2,4}$]{
  \begin{tikzpicture}[scale=2.5]
  \draw[thin, gray!50](0,0)--(1,0);
  \filldraw (.2, 0) node {} circle  (1pt);
  \filldraw (.4, 0) node {} circle  (1pt);
  \filldraw (.6, 0) node {} circle  (1pt);
  \filldraw (.8, 0) node {} circle  (1pt);
  \label{subfig:U_2,4}
  \end{tikzpicture}}
\hspace{30pt}
\subfigure[$L_2$]{
  \begin{tikzpicture}[scale=2.5]
  \draw[thick](0,0)--(1,0);
  \filldraw (.2, 0) node {} circle  (1pt);
  \filldraw (.8, 0) node {} circle  (1pt);
  \node at (.5,.05) {\footnotesize $a$};
  \node at (.2,.1) {\footnotesize $b$};
  \node at (.8,.09) {\footnotesize $c$};
  \label{subfig:L_2}
  \end{tikzpicture}}
\hspace{30pt}
\subfigure[$S_2$]{
  \begin{tikzpicture}[scale=2.5]
  \draw[thick](0,.025)--(1,.025);
  \draw[thick](-.2,0)--(.8, 0);
  \label{subfig:S_2}
  \end{tikzpicture}}
\caption{The three polymatroids $\rho$ with $ M_\rho=U_{2,4}$.   All
  are excluded minors for $\mathcal{P}_{U_{2,4}}$.   The gray line in part
  \subref{subfig:U_2,4} represents the flat spanned by the points, not
  an element of $\rho$.   The lines in part \subref{subfig:S_2} are
  parallel.}
\label{fig:natural}
\end{figure}

\section{The $k$-natural matroid and $k$-duality}\label{section:k-natural}

There are multiple polymatroid counterparts of duality.  We focus on
$k$-duality.  Given a $k$-polymatroid $(E,\rho)$, its $k$\emph{-dual}
is the $k$-polymatroid $(E, \rho^*)$  defined by
$$\rho^*(X)=k|X|-\rho(E)+\rho(E-X)$$ for $X\subseteq E$.   When
we refer to $\rho^*$, we will make clear what $k$ is.  Some important
and desirable properties make this notion of duality particularly
natural.  For instance, Whittle \cite{GeoffDual} showed that the map
$\rho\mapsto \rho^*$ is the only involution on the class of
$k$-polymatroids that switches deletion and contraction, i.e.,
$(\rho_{\del e})^*=(\rho^*)_{/e}$ and
$(\rho_{/e})^*=(\rho^*)_{\del e}$ for all $e\in E$.  (The counterpart
for matroids was shown earlier by Kung \cite{JPSK}.)  It follows that
if a minor-closed class of $k$-polymatroids is closed under
$k$-duality, then so is its set of excluded minors.

We prepare for our work on the excluded minors for the three classes
of $2$-polymatroids discussed in Section \ref{sec:classes} by first
proving Theorem \ref{thm:dual closure of P} below, which implies that
each of these classes is closed under $2$-duality and, therefore, so
is its set of excluded minors.  For this we introduce the $k$-natural
matroid of $\rho$.  Let $(E, \rho)$ be a $k$-polymatroid.  The set
$X_e$, for $e\in E$, is as in the definition of the natural matroid
$M_\rho=(X_E, r)$ of $\rho$.  Let $M^k_\rho=(Y_E, r_k)$, the
\emph{$k$-natural matroid} of $\rho$, be the matroid obtained by the
following iterated principal extensions of $M_\rho$: for each
$e\in E$, freely add a set $U_e$ of $k-\rho(e)$ elements to $X_e$.
For $e\in E$, set $Y_e=X_e\cup U_e$, and for $A\subseteq E$, set
$$Y_A=\bigcup_{e\in A}Y_e.$$ Note that if 
$\rho(e)=k$ for all $e\in E$, then $M_\rho=M^k_\rho$.  Observe that
Lemma~\ref{lemma:shownatural} extends to $k$-natural matroids, as
follows.

\begin{lemma}\label{lemma:showknatural}
  Let $\rho$, $E$,  $Y_e$, and $Y_A$ be as above.  A matroid $M$
  on $Y_E$ is $M^k_\rho$ if and only if each set $Y_e$ is a set of clones
  and $r_M(Y_A)=\rho(A)$ for all $A\subseteq E$.
\end{lemma}

An observation that we made about the natural matroid extends
immediately to the $k$-natural matroid: if $\rho$ and $\rho'$ are
polymatroids on $E$ with $M^k_\rho=M^k_{\rho'}$ where, for each
$e\in E$, the corresponding set $Y_e$ is the same in both $k$-natural
matroids, then $\rho=\rho'$.

We use a result from \cite{natural matroid} on types of bases and
develop related ideas for the $k$-natural matroid.  Let $\rho$ be a
polymatroid on $E=[n]$ and let $M_\rho=(X_E,r)$ be its natural
matroid.  For $V\subseteq X_E$, its \emph{type} in $M_\rho$ is the
vector $\textbf{v}=(v_1, v_2, \ldots, v_n)$ where $v_i=|V\cap X_i|$.
Since each set $X_i$ is a set of clones, if $V, V'\subseteq X_E$ have
the same type, then $r(V)=r(V')$; thus, $V$ is independent in $M_\rho$
if and only if $V'$ is.  The same holds for bases, circuits, and
flats.

For $V\subseteq Y_E$, its \emph{type} in the $k$-natural matroid
$M_\rho^k$ is $\textbf{v}=(v_1, v_2, \ldots, v_n)$ where
$v_i=|V\cap Y_i|$.  Since the elements of $U_i$ are freely added to
the set $X_i$ of clones, $Y_i$ is also a set of clones.  Hence, as in
the natural matroid, sets in the $k$-natural matroid of the same type
have the same rank.  If $V\subseteq E(M_\rho)$, then
$V\subseteq E(M_\rho^k)$ and its type is the same in $M_\rho$ as in
$M_\rho^k$.  If it is clear whether we are referring to $M_\rho$ or
$M_\rho^k$, we will just say that $V$ has type $\textbf{v}$.

For any $n$-tuple $\textbf{u}$, each entry of which is an integer
between $0$ and $k$, the $n$-tuples $\textbf{u}$ and
$(k, k, \ldots, k)-\textbf{u}$ are \emph{$k$-complements}.  The
following result, from \cite{natural matroid}, generalizes the fact
that the bases of a matroid $M$ are the complements of the bases of
its dual, $M^*$.

\begin{lemma}\label{ref:complementtypes}
  Let $\rho$ be a $k$-polymatroid and $\rho^*$ be its $k$-dual.  If
  $\mathcal{B}$ is the set of types of bases of $M_\rho$, then the set
  of types of bases of $M_{\rho^*}$ is
  $\{(k, k,\ldots, k)-\mathbf{u}\,:\,\mathbf{u}\in\mathcal{B}\}$.  That
  is, the types of bases of $M_{\rho^*}$ are the $k$-complements of
  the types of bases of $M_\rho$.
\end{lemma}

Note that $M_{\rho^*}$ need not be dual to $M_\rho$; indeed, $M_\rho$
and $M_{\rho^*}$ may not even have the same number of elements.  The
following example of this is illustrated in Figure
\ref{fig:k-natural}.

\begin{figure}
\centering
\subfigure[$M_{L_2}^2$]{
\begin{tikzpicture}[scale=2]
  \draw[thin, gray!50](0,0)--(2,0);
  \filldraw (.4, 0) node[above] {\footnotesize $a_1$} circle  (1pt);
  \filldraw (1.6,0) node[above] {\footnotesize $a_2$} circle  (1pt);
  \filldraw (.8, .03) node[above] {\footnotesize $b_1$} circle  (1pt);
  \filldraw (.8, -.03) node[below] {\footnotesize $b_2$} circle  (1pt);
  \filldraw (1.2, .03) node[above] {\footnotesize $c_1$} circle  (1pt);
  \filldraw (1.2, -.03) node[below] {\footnotesize $c_2$} circle  (1pt);
  \label{subfig:M_{L_2^2}}
\end{tikzpicture}}
	
\subfigure[$A_4$]{
  \begin{tikzpicture}[scale=2]
    \draw[thick](0,0)--(0,1);
    \draw[thick](-.7,-.25)--(-.7,.75);
    \draw[thick](.7,-.25)--(.7,.75);
    \draw[thin, gray!50](0,1)--(-.7,.75);
    \draw[thin, gray!50](0,0)--(-.7,-.25);
    \draw[thin, gray!50](0,1)--(.7,.75);
    \draw[thin, gray!50](0,0)--(.7,-.25);
    \node at (.1,.5) {\footnotesize $a$};
    \node at (-.8,.25) {\footnotesize $b$};
    \node at (.8,.25) {\footnotesize $c$};
    \label{subfig:A_4}
  \end{tikzpicture}}
\hspace{30pt}
\subfigure[$M_{A_4}=M_{A_4}^2$]{
  \begin{tikzpicture}[scale=2]
    \draw[thin, gray!50](0,0)--(0,1);
    \draw[thin, gray!50](-.7,-.25)--(-.7,.75);
    \draw[thin, gray!50](.7,-.25)--(.7,.75);
    \draw[thin, gray!50](0,1)--(-.7,.75);
    \draw[thin, gray!50](0,0)--(-.7,-.25);
    \draw[thin, gray!50](0,1)--(.7,.75);
    \draw[thin, gray!50](0,0)--(.7,-.25);
    \filldraw (0, .3) node[right] {\footnotesize $a_2$} circle  (1pt);
    \filldraw (0,.6) node[right] {\footnotesize $a_1$} circle  (1pt);
    \filldraw (-.7, 0.08) node[left] {\footnotesize $b_2$} circle  (1pt);
    \filldraw (-.7,.42) node[left] {\footnotesize $b_1$} circle  (1pt);
    \filldraw (.7, .08) node[right] {\footnotesize $c_2$} circle  (1pt);
    \filldraw (.7,.42) node[right] {\footnotesize $c_1$} circle  (1pt);
    \label{subfig:M_A_4}
  \end{tikzpicture}}
\caption{Part \subref{subfig:M_{L_2^2}} shows the $2$-natural matroid
  of the $2$-polymatroid $L_2$ shown in Figure \ref{fig:natural}.  Part
  \subref{subfig:A_4} shows the $2$-dual of $L_2$, denoted $A_4$.  Part
  \subref{subfig:M_A_4} shows the $2$-natural matroid of $A_4$, which
  is also its natural matroid.}
\label{fig:k-natural}
\end{figure}
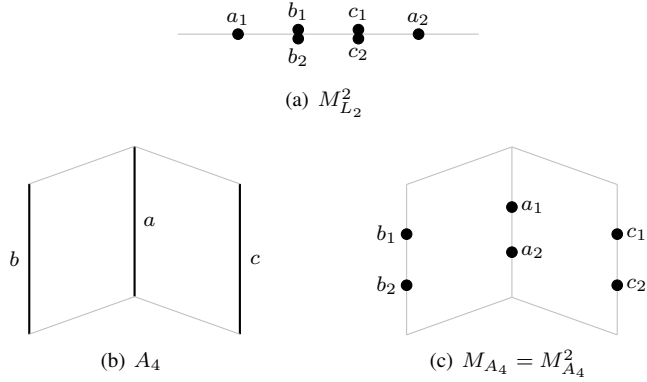

\begin{example}\label{ex:k-natural}
  Let $L_2$ be the $2$-polymatroid shown in Figure \ref{fig:natural}.
  The $2$-dual of $L_2$ consists of three lines in rank $4$ where $a$
  is coplanar with each of $b$ and $c$, and $b$ is skew to $c$.  Let
  $A_4$ denote this polymatroid.  With the order $a<b<c$ on the
  elements of $L_2$ and $A_4$, the types of bases of $M_{L_2}$ are
  $(2,0,0)$, $(1,1,0)$, $(1,0,1),$ and $(0,1,1)$, and those of
  $M_{A_4}$ are their $2$-complements, $(0,2,2)$, $(1,1,2)$,
  $(1,2,1),$ and $(2,1,1)$.  However, $M_{L_2}$ is not dual to
  $M_{A_4}$ since $|E(M_{L_2})|=4$ and $|E(M_{A_4})|=6$.
	
  Consider the $2$-natural matroids.  Now $M_{L_2}^2$ just adds points
  parallel to points in $M_{L_2}$, while $M_{A_4}^2=M_{A_4}$ since all
  elements of $A_4$ are lines.  While $M_{L_2}^2$ has more bases than
  $M_{L_2}$, the types of bases are the same, so the types of bases of
  $M_{L_2}^2$ are $2$-complements of those of $M_{A_4}^2$.  Also, one
  can check that $M_{L_2}^2$ is the dual of $M_{A_4}^2$.
\end{example}

The next lemma shows, for any $k$-polymatroid $\rho$, that $M_\rho$
has the same basis types as $M_\rho^k$ and that $M_\rho^k$ is dual to
$M_{\rho^*}^k$.

\begin{lemma}\label{lemma:dual of k-natural}
  Let $\rho$ be a $k$-polymatroid on $E=[n]$ and $\rho^*$ be its
  $k$-dual.  Then
  \begin{itemize}
  \item[(1)] $\mathbf{u}$ is the type of a basis of $M_\rho$ if and
    only if $\mathbf{u}$ is the type of a basis of $M_\rho^k$,
  \item[(2)] the types of bases of $M^k_{\rho^*}$ are the
    $k$-complements of the types of bases of $M^k_\rho$, and
  \item[(3)] the $k$-natural matroid of $\rho$ is dual to that of
    $\rho^*$, i.e., $(M_\rho^k)^*=M_{\rho^*}^k$.
  \end{itemize}
\end{lemma}

\begin{proof}
  Each basis of $M_\rho$ is a basis of $M_\rho^k$, so each type of a
  basis of $M_\rho$ is the type of a basis of $M_\rho^k$.  For the
  converse, let $\mathbf{u}$ be the type of a basis $B$ of $M_\rho^k$.
  Now $u_i\leq \rho(i)$ for all $i\in E$.  Since $Y_i$ is a set of
  clones, we may assume that $B\subseteq X_E$.  Thus, $B$ is a basis
  of $M_\rho$, and so $\mathbf{u}$ is the type of a basis of $M_\rho$.
  Thus, part (1) holds.

  Part (2) follows from part (1) and Lemma \ref{ref:complementtypes}.
  We can let $M^k_\rho$ and $M^k_{\rho^*}$ have the same ground set
  (since $|Y_i|=k$ for all $i\in E$), in which case it follows that
  the complements of the bases of $M^k_\rho$ are the bases of
  $M^k_{\rho^*}$, so part (3) holds.
\end{proof}

Using this lemma, we next obtain the main result of this section,
which applies to the three classes of $2$-polymatroids discussed in
Section \ref{sec:classes}.

\begin{thm}\label{thm:dual closure of P}
  Let $\mathcal{C}$ be a minor-closed, dual-closed class of matroids
  such that all excluded minors for $\mathcal{C}$ are simple.  Then
  the class of $2$-polymatroids whose natural matroids are in
  $\mathcal{C}$ is closed under $2$-duality, as is its set of excluded
  minors.
\end{thm}

\begin{proof}
  For a $2$-polymatroid $\rho$, each element added to $M_\rho$ to get
  $M^2_\rho$ is either a loop or parallel to an element of $M_\rho$,
  so $M^2_\rho\in\mathcal{C}$ if and only if $M_\rho\in\mathcal{C}$.
  The first assertion now follows from Lemma \ref{lemma:dual of
    k-natural}.  The second assertion follows from the first.
\end{proof}

Thus, for instance, $A_4$, the $2$-dual of $L_2$, is also an excluded
minor for $\mathcal{P}_{U_{2,4}}$.

While the counterpart of Theorem \ref{thm:dual closure of P} might
hold for $k$-polymatroids with $k>2$, the proof is limited to $k=2$.
Consider the polymatroid $\rho$ consisting of a line with a point on
it, and its $3$-dual, $\rho^*$, which consists of a plane and a line
in rank $4$.  While both $\rho$ and $\rho^*$ are binary-natural, the
$3$-natural matroid of $\rho$ is $U_{2,4}$ with one point tripled,
which is not binary.

The next result follows from Lemma \ref{lemma:dual of k-natural} and
the properties of the $k$-dual noted above.

\begin{lemma}\label{lemma:minors in k-natural}
  Let $(E,\rho)$ be a $k$-polymatroid and $(E,\rho^*)$ be its
  $k$-dual.  If $e\in E$, then
  $M_{\rho_{\setminus e}}^k=(M_{\rho^*_{/e}}^k)^*$ and
  $M_{\rho_{/e}}^k=(M_{\rho^*_{\setminus e}}^k)^*$.
\end{lemma}

\begin{proof}
  We have
  $M_{\rho_{\setminus e}}^k= (M_{(\rho_{\setminus e})^*}^k)^*=
  (M_{\rho^*_{/e}}^k)^*$.  The other equality follows similarly.
\end{proof}

Observe that the counterpart of the deletion result in Lemma
\ref{lem:naturalminor} holds for the $k$-natural matroid:
$M_{\rho_{\setminus e}}^k=M_\rho^k\setminus Y_e$.  While for the
natural matroid, $M_{\rho_{/ e}}$ is a deletion of $M_\rho/ X_e$, the
situation is simpler for the $k$-natural matroid:
$M_{\rho_{/e}}=M_\rho^k/Y_e$ for all $e\in E$.

We end this section with some consequences of Lemma \ref{lemma:dual of
  k-natural} that are not used in this paper.  By that lemma, the
circuits of $M^k_\rho$ are the complements of the hyperplanes of
$M^k_{\rho^*}$.  As noted in \cite{natural matroid}, this property can
fail in the natural matroid.  Unlike the types of bases, the types of
circuits of $M^k_\rho$ may differ from those of circuits of $M_\rho$,
but in a predictable way; the types of circuits of $M^k_\rho$ consist
of those of circuits of $M_\rho$ and the vectors
$(\rho(i)+1)\mathbf{e}_i$, for $i\in E$, where $\rho(i)<k$ and
$\mathbf{e}_i$ is the $i$-th unit vector.

\begin{ex:2}
  Consider $L_2$ and $A_4$ as in Figures \ref{fig:natural} and
  \ref{fig:k-natural}.  The types of circuits of $M_{L_2}$ are
  $(2,1,0)$, $(2,0,1)$, and $(1,1,1)$, while those of hyperplanes of
  $M_{A_4}$ are $(2,2,0)$, $(2,0,2)$, $(0,1,2)$, $(0,2,1)$, and
  $(1,1,1)$.  Precisely the last three types of hyperplanes are
  $2$-complements of the types of circuits.  The additional types of
  circuits in $M_{L_2}^2$ are $(0,0,2)$ and $(0,2,0)$, which are the
  $2$-complements of the first two types of hyperplanes.  Since
  $M_{A_4}^2=M_{A_4}$, the types of hyperplanes are the same, so the
  $2$-complementary property between circuits and hyperplanes holds.
\end{ex:2}

\section{Compression and its role for excluded minor
  problems}\label{section:compression}

The polymatroid operation of compression was defined in \cite{wheels
  and whirls}.  We recall this operation after explaining why it
arises naturally here.

Let $(E,\rho)$ be an excluded minor for $\mathcal{P}_{U_{2,4}}$, so
$M_\rho$ has a $U_{2,4}$-minor.  Consider $e\in E$.  Since
$M_\rho\setminus X_e=M_{\rho_{\setminus e}}$ and
$M_{\rho_{\setminus e}}\in \mathcal{P}_{U_{2,4}}$, deleting $X_e$
eliminates all $U_{2,4}$-minors of $M_\rho$.  Also,
$M_{\rho_{/e}}\in \mathcal{P}_{U_{2,4}}$ and $M_\rho/X_e$ is
$M_{\rho_{/e}}$ or an extension of it by loops and elements parallel
to those in $M_{\rho_{/e}}$, so contracting $X_e$ eliminates all
$U_{2,4}$-minors of $M_\rho$.  Thus, if $\rho(e)=1$, so $X_e=\{e_1\}$,
then $e_1$ is in each $U_{2,4}$-minor of $M_\rho$.  Also, if
$\rho(e)=2$, so $X_e=\{e_1, e_2\}$, then, to get a $U_{2,4}$-minor of
$M_\rho$, either we
\begin{enumerate}
\item delete one of $e_1$ and  $e_2$ and contract the other, or
\item have at least one of $e_1$ and $e_2$ in the $U_{2,4}$-minor.
\end{enumerate}
Case (2) can occur for at most four elements $e$.  If case (1) never
occurs, then $|E|\leq 4$ since each $U_{2,4}$-minor of $M_\rho$
contains at least one element of each set $X_e$.  If case (1) applies
to $e$, we would like a $2$-polymatroid $\rho'$ such that
$M_{\rho'}=M_\rho/ e_1\setminus e_2$.  If such a $\rho'$ always exists
and is an excluded minor for $\mathcal{P}_{U_{2,4}}$, then we can
reduce $\rho$ to an excluded minor for which case (1) does not occur.
Compression gives this polymatroid $\rho'$.

For a polymatroid $\rho$ on $E$ and $e\in E$, the \emph{compression}
of $\rho$ by $e$ is the polymatroid $\rho_{\downarrow e}$ on $E-e$
obtained by freely adding a point $x$ to $e$, then contracting $x$ and
deleting $e$.  If $e$ is a loop, then
$\rho_{\downarrow e}=\rho_{\setminus e}$.  If $\rho(e)>0$, then
compression can be defined, equivalently, by
\begin{equation}\label{eq:compression}
(\rho_{\downarrow e})(X)=
\begin{cases}
\rho(X)-1, & \text{if } X \text{ spans } e \text{ in } \rho,\\
\rho(X), & \text{otherwise,}
\end{cases}
\end{equation}
for all $X\subseteq E-e$. If $e$ is a point, then
$\rho_{\downarrow e}=\rho_{/e}$.

\begin{lemma}\label{lem: natcomp}
  For a polymatroid $(E,\rho)$, fix $e\in E$ with $\rho(e)>0$ and fix
  $e_1\in X_e$.  Let the set $Y$ consist of exactly one element in
  each set $X_f$ for which $f\ne e$ and $\rho(\{e,f\})=\rho(f)$.  Then
  the natural matroid $M_{\rho_{\downarrow e}}$ is
  $M_\rho/e_1\setminus \bigl((X_e-e_1)\cup Y\big)$.  Likewise, the
  $k$-natural matroid $M^k_{\rho_{\downarrow e}}$ is
  $M^k_\rho/e_1\setminus (X_e-e_1)$.
\end{lemma}

\begin{proof} Set
  $N_e= M_\rho/e_1\setminus ( \bigl(X_e-e_1)\cup Y\big)$ and let
  $X'_A=X_A-Y$ for all $A\subseteq E-e$.  For the first assertion, by
  Lemma \ref{lemma:shownatural} it suffices to show that
  \begin{itemize}
  \item[(i)] for each $g\in E-e$, the set $X'_g$ is a set of clones of
    $N_e$, and
  \item[(ii)] $r_{N_e}(X'_A)=\rho_{\downarrow e}(A)$ for all
    $A\subseteq E-e$.
  \end{itemize}
  Property (i) holds since $X_g$ is a set of clones of $M_\rho$ and
  clones in a matroid remain clones in each minor that contains them.
  Property (ii) follows from the definition of $\rho_{\downarrow e}$.
  The second assertion follows similarly, using Lemma
  \ref{lemma:showknatural}.
\end{proof}

If $e$ is a line in a $2$-polymatroid, then having no element span $e$
is equivalent to having no line parallel to $e$, so the lemma
specializes to the corollary below.

\begin{cor}\label{cor: natural matroid of compression}
  If $\rho$ is a $2$-polymatroid and $e$ is a line of $\rho$, then
  $M^2_{\rho_{\downarrow e}}=M^2_\rho/e_1\setminus e_2$.  If, in
  addition, no line of $\rho$ is parallel to $e$, then
  $M_{\rho_{\downarrow e}}=M_\rho/e_1\setminus e_2$.
\end{cor}

We next give an instance where compression and taking the $2$-dual
commute.

\begin{lemma}\label{lem:commute}
  Let $(E,\rho)$ be a $2$-polymatroid.  For $e\in E$, if
  $\rho(e)=2=\rho^*(e)$, then
  $(\rho_{\downarrow e})^* = (\rho^*)_{\downarrow e}$.
\end{lemma}

\begin{proof}
  It suffices to show that
  $M^2_{(\rho_{\downarrow e})^*} = M^2_{(\rho^*)_{\downarrow e}}$, and
  that holds since, by Lemma \ref{lemma:dual of k-natural} and
  Corollary \ref{cor: natural matroid of compression},
  \begin{align*}
    M^2_{(\rho_{\downarrow e})^*}
    & = (M^2_{\rho_{\downarrow e}})^*\\
    & = (M^2_\rho /e_1\setminus e_2)^*\\
    & = (M^2_\rho)^* \setminus e_1/ e_2 \\
    & = M^2_{\rho^*} \setminus e_1/ e_2 \\
    & =  M^2_{(\rho^*)_{\downarrow e}}.\qedhere
  \end{align*}
    
\end{proof}

In Section \ref{sec:classes} and Figure \ref{fig:natural}, we saw the
excluded minor $S_2$ for $\mathcal{P}_{U_{2,4}}$.  This is the first
in an infinite sequence of excluded minors that we define next and
show are related by compression.  For $n\geq 2$, let $S_n$ be
$(E, \rho)$ where $|E|=n$ and
\begin{equation}\label{eq:spikelike}
  \rho(X) =
  \begin{cases}
    |X|+1, & \text{if } 0<|X|<n,\\
    |X|, & \text{otherwise.}
  \end{cases}
\end{equation}
It is easy to check that $S_n$ is its own $2$-dual, that is,
$S^*_n=S_n$.  For $n>2$, the natural matroid of $S_n$ is the tipless,
free $n$-spike.

\begin{example}\label{ex: spikelike compression}
  For $n\geq 2$, let $S_{n+1}=(E, \rho)$ and let $e\in E$.  Consider
  the compression $(S_{n+1})_{\downarrow e}$.  Certainly
  $\rho_{\downarrow e}(\emptyset)=0$.  The only subset of $E-e$ that
  spans $e$ is $E-e$ itself, so $\rho_{\downarrow e}(X)=\rho(X)=|X|+1$
  for $\emptyset \subset X\subset E-e$, and
  $\rho_{\downarrow e}(E-e)=\rho(E-e)-1=|E-e|$.  Hence,
  $(S_{n+1})_{\downarrow e}$ is isomorphic to $S_n$.
\end{example}

Using Corollary \ref{cor: natural matroid of compression}, we show
that $S_n$, for $n\geq 2$, is an excluded minor for
$\mathcal{P}_{U_{2,4}}$.

\begin{prop}\label{prop:spikelike}
  For $n\geq 2$, the $2$-polymatroid $S_n$ is an excluded minor for
  $\mathcal{P}_{U_{2,4}}$.
\end{prop}

\begin{proof}
  Now $S_2$ is an excluded minor since $M_{S_2}$ is $U_{2,4}$.  Since
  $(S_n)_{\downarrow e}=S_{n-1}$ for all $e\in S_n$, by induction and
  Corollary \ref{cor: natural matroid of compression}, each $M_{S_n}$
  has a $U_{2,4}$-minor.  To see that all proper minors of $S_n$ are
  binary-natural, take $e\in S_n$.  The contraction $(S_n)_{/e}$ is
  the binary matroid $U_{n-2,n-1}$, and
  $(S_n)_{\setminus e}=\bigl((S_n^*)_{/e}\bigr)^*
  =\bigl((S_n)_{/e}\bigr)^* $, which is in $\mathcal{P}_{U_{2,4}}$
  by Theorem \ref{thm:dual closure of P}.
\end{proof}	

We note that the natural matroid of $(S_n)_{\setminus e}$ is the
parallel connection of $n-1$ copies of $U_{2,3}$ at a common base
point, with that point deleted.

The next lemma, which plays key roles in both Sections
\ref{section:main theorem} and \ref{section:K_4}, simplifies the
problem of identifying excluded minors by allowing us to consider
compression separately.  For concreteness, we first discuss the result
and its significance for $\mathcal{P}_{U_{2,4}}$.  In that case, this
lemma says that, for any excluded minor $\rho$ for
$\mathcal{P}_{U_{2,4}}$ and line $e$ of $\rho$, the compression
$\rho_{\downarrow e}$ is also an excluded minor if and only if
$\rho_{\downarrow e}\not\in\mathcal{P}_{U_{2,4}}$.  Let
$\mathcal{R}_{U_{2,4}}$ be the set of excluded minors $\rho$ for
$\mathcal{P}_{U_{2,4}}$ for which
$\rho_{\downarrow e}\in\mathcal{P}_{U_{2,4}}$ for each line $e$ of
$\rho$.  It follows that if $\rho' \not \in \mathcal{R}_{U_{2,4}}$ is
an excluded minor for $\mathcal{P}_{U_{2,4}}$, then some sequence of
compressions of lines starting with $\rho'$ yields an excluded minor
in $\mathcal{R}_{U_{2,4}}$, and all polymatroids in that sequence are
excluded minors.  This gives a two-part strategy for finding the
excluded minors for $\mathcal{P}_{U_{2,4}}$.  First we find
$\mathcal{R}_{U_{2,4}}$.  Then, for each
$\rho\in\mathcal{R}_{U_{2,4}}$, we find all excluded minors $\rho'$
for $\mathcal{P}_{U_{2,4}}$ for which $\rho'_{\downarrow e}=\rho$ for
some line $e$ of $\rho'$.  We then repeat this step on all new
excluded minors that we find until this process terminates or we have
characterized all infinite families of excluded minors for
$\mathcal{P}_{U_{2,4}}$ obtained from this process.

\begin{lemma}\label{lemma: compression of excluded minor}
  Let $\mathcal{C}$ be a minor-closed class of matroids and
  $\mathcal{C}'_2$ be the class of $2$-polymatroids whose natural
  matroids are in $\mathcal{C}$.  Let $(E,\rho)$ be an excluded minor
  for $\mathcal{C}'_2$ and fix $e\in E$ with $\rho(e)=2$. Then
  $\rho_{\downarrow e}$ is an excluded minor for $\mathcal{C}'_2$ if
  and only if $\rho_{\downarrow e}\not\in\mathcal{C}'_2$.
\end{lemma}

\begin{proof}
  One implication is immediate.  For the converse, assume that
  $\rho_{\downarrow e}\not\in\mathcal{C}'_2$.  We claim that no line
  of $\rho$ is parallel to $e$.  Assume, to the contrary, that the
  line $f$ is parallel to $e$.  Let $X_e=\{e_1,e_2\}$ and
  $X_f=\{f_1,f_2\}$.  Now $e_1,e_2, f_1,f_2$ are clones in $M_\rho$.
  By Lemma \ref{lem: natcomp}, we get $M_{\rho_{\downarrow e}}$ from
  $M_\rho$ by contracting $e_1$ and deleting $e_2$, $f_1$, and perhaps
  more elements.  In particular,
  $M_{\rho}\setminus\{e_2,f_1\}\not\in\mathcal{C}$.  Since $e_1$ and
  $f_1$ are clones in $M_\rho$, this implies that
  $M_{\rho}\setminus\{e_1,e_2\}\not\in\mathcal{C}$, and so
  $\rho_{\setminus e}\not\in\mathcal{C}$, which contradicts $\rho$
  being an excluded minor.  Thus, no line of $\rho$ is parallel to
  $e$.
   
  Fix $f\in E-e$.  To show that $(\rho_{\downarrow e})_{\setminus f}$
  is in $\mathcal{C}'_2$, it suffices to show that
  $M_{(\rho_{\downarrow e})_{\setminus f}}$ is a minor of
  $M_{\rho_{\setminus f}}$, which is in $\mathcal{C}$, and that holds
  since, by Lemma \ref{lem:naturalminor} and Corollary \ref{cor:
    natural matroid of compression},
  \begin{align*}
    M_{(\rho_{\downarrow e})_{\setminus f}}  
    & = M_{(\rho_{\downarrow e})} \setminus X_f \\
    & = M_\rho / e_1\setminus e_2\setminus X_f \\
    & = M_{\rho_{\setminus f}}/ e_1\setminus e_2.
  \end{align*}
  We next show that $(\rho_{\downarrow e})_{/f}$ is in
  $\mathcal{C}'_2$.  Set $E'=E-\{e,f\}$.  Lemma \ref{lem:naturalminor}
  and Corollary \ref{cor: natural matroid of compression} give
  $$M_{(\rho_{\downarrow e})_{/f}} = \bigl(M_\rho/(X_f\cup
  e_1)\bigr)|U$$ where $U\subseteq X_{E'}$ and
  $|U\cap X_g|= (\rho_{\downarrow e})_{/f}(g)$ for all $g\in E'$.   Likewise,
  $$M_{\rho_{/f}}/e_1\setminus e_2 =\bigl( M_\rho/(X_f\cup
  e_1)\bigr)|V$$ where $V\subseteq X_{E'}$ and
  $|V\cap X_g|= \rho_{/f}(g)$ for all $g\in E'$.  Thus, it will follow
  that $M_{(\rho_{\downarrow e})_{/f}}$ is isomorphic to a minor of
  $M_{\rho_{/f}}$, which is in $\mathcal{C}$ and so gives
  $M_{(\rho_{\downarrow e})_{/f}}\in\mathcal{C}$, if we show that
  $(\rho_{\downarrow e})_{/f}(g)\leq \rho_{/f}(g)$, that is,
  $\rho_{\downarrow e}(\{f,g\})-\rho_{\downarrow e}(f) \leq
  \rho(\{f,g\})- \rho(f)$, for all $g\in E'$.  This holds since
  $\rho_{\downarrow e}(\{f,g\}) \leq \rho(\{f,g\})$ and
  $\rho_{\downarrow e}(f)=\rho(f)$ since $f$ is not parallel to $e$.
  Thus, all proper minors of $\rho_{\downarrow e}$ are in
  $\mathcal{C}'_2$, so $\rho_{\downarrow e}$ is an excluded minor for
  $\mathcal{C}'_2$.
\end{proof}

Let $\mathcal{C}$ be a minor-closed, dual-closed class of matroids,
all excluded minors for which are simple.  Let $\mathcal{C}'_2$ be the
class of $2$-polymatroids whose natural matroids are in $\mathcal{C}$,
and let $\mathcal{R}_{\mathcal{C}'_2}$ be the set of excluded minors
$(E,\rho)$ for $\mathcal{C}'_2$ such that
$\rho_{\downarrow e}\in\mathcal{C}'_2$ for all $e\in E$.  We next show
that $\mathcal{R}_{\mathcal{C}'_2}$ is closed under $2$-duality.  In
particular, this applies to the three classes of $2$-polymatroids
discussed in Section \ref{sec:classes}.

\begin{lemma}\label{lemma:R 2-duality}
  If $\mathcal{C}$, $\mathcal{C}'_2$, and
  $\mathcal{R}_{\mathcal{C}'_2}$ are as above, then the set
  $\mathcal{R}_{\mathcal{C}'_2}$ is closed under $2$-duality.
\end{lemma}

\begin{proof}
  Let $(E,\rho)$ be in $\mathcal{R}_{\mathcal{C}'_2}$.  By Theorem
  \ref{thm:dual closure of P}, its $2$-dual $\rho^*$ is an excluded
  minor for $\mathcal{C}'_2$.  To show that
  $\rho^*\in\mathcal{R}_{\mathcal{C}'_2}$, consider $e\in E$ with
  $\rho^*(e)=2$.  If $\rho^*(E-e)<\rho^*(E)$, then
  $(\rho^*)_{\downarrow e}=(\rho^*)_{\setminus e}$, and so
  $(\rho^*)_{\downarrow e} \in \mathcal{C}'_2$.  Now assume that
  $\rho^*(E- e)=\rho^*(E)$.  Thus, $\rho(e)=2$, so Lemma
  \ref{lem:commute} gives
  $ (\rho^*)_{\downarrow e}=(\rho_{\downarrow e})^*$.  Since
  $\rho_{\downarrow e} \in \mathcal{C}'_2$ (since
  $\rho\in\mathcal{R}_{\mathcal{C}'_2}$), we get
  $(\rho^*)_{\downarrow e} \in \mathcal{C}'_2$.  Thus,
  $\rho^*\in\mathcal{R}_{\mathcal{C}'_2}$.
\end{proof}

\section{The excluded minors for the class of binary-natural
  $2$-polymatroids}\label{section:main theorem}

We give the excluded minors for the class $\mathcal{P}_{U_{2,4}}$ of
binary-natural $2$-polymatroids and prove that the list is complete.
Some of the excluded minors that we have not yet discussed are
illustrated in Figure \ref{figure: excluded minors}.  Each excluded
minor is connected and contains no loops, no parallel points, and,
except for $S_2$, no parallel lines.

\begin{thm}\label{thm: main theorem}
  The excluded minors for $\mathcal{P}_{U_{2,4}}$ are
  \begin{itemize}
  \item $U_{2,4}$;
  \item $L_2$: a line with two points on it;
  \item $S_n$, for $n\geq 2$: these are defined before, and treated
    in, Proposition \ref{prop:spikelike};
  \item $A_3$: two coplanar lines with a point on exactly one of them;
  \item $B_3$: in rank $3$, a line and three points, none on the line,
    having a $U_{2,3}$-restriction;
  \item $A_4$: the $2$-dual of $L_2$; three lines in rank $4$, exactly
    one pair of which is skew;
  \item $B_4$: two skew lines and two points in rank $4$ where each
    line is coplanar with the pair of points but spans neither;
  \item $A_5$: the $2$-dual of $B_3$; three lines and a point in rank
    $5$ where the point does not lie on any line, and each pair of
    lines is skew and spans the point; and
  \item $A_6$: the $2$-dual of $U_{2,4}$; four lines in rank $6$ with
    $\rho(X)=|X|+2$ if $|X|\in\{2,3\}$.
  \end{itemize}
  Also, all excluded minors except $S_n$, for $n\geq 3$, belong to
  $\mathcal{R}_{U_{2,4}}$, the set of excluded minors $\rho$ for which
  $\rho_{\downarrow e}\in\mathcal{P}_{U_{2,4}}$ for each line $e$ of
  $\rho$.
\end{thm}

Note that we get $B_3$ from $A_3$, and $B_4$ from $A_4$, by replacing
a particular line by two points placed freely on it.  By Theorem
\ref{thm:dual closure of P} and Lemma \ref{lemma:R 2-duality},
$\mathcal{P}_{U_{2,4}}$, its set of excluded minors, and
$\mathcal{R}_{U_{2,4}}$ are closed under $2$-duality.  The excluded
minors $S_n$, for $n\geq 2$, $A_3$, and $B_4$ are self-$2$-dual.  It
is easy to check that every polymatroid listed in Theorem \ref{thm:
  main theorem} is an excluded minor for $\mathcal{P}_{U_{2,4}}$, and
that each except $S_n$ with $n\geq 3$ has no compression that is an
excluded minor.

Before proving our main theorem, we give a useful lemma involving the
polymatroids $Z_3$ and $Z_{2,2}$ as in Figure \ref{fig:exceptions},
which are binary-natural.

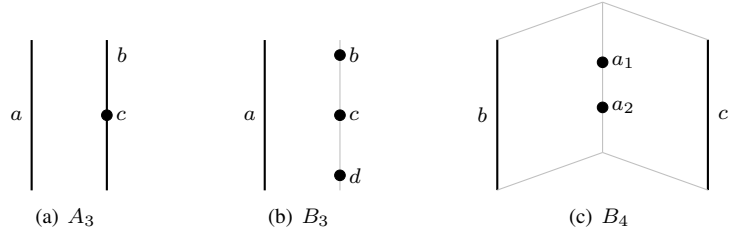
\begin{figure}
  \centering
  \subfigure[$A_3$]{
    \begin{tikzpicture}[scale=2]
      \draw[thick](0,0)--(0,1);
      \draw[thick](.5,0)--(.5,1);      
      \filldraw (.5,.5) node[right] {\footnotesize $c$} circle (1pt);      
      \node at (-.1,.5) {\footnotesize $a$};      
      \node at (.6,.9) {\footnotesize $b$};      
      \label{subfig:A_3}
    \end{tikzpicture}}
  \hspace{30pt}
  \subfigure[$B_3$]{
    \begin{tikzpicture}[scale=2]
      \draw[thick](0,0)--(0,1);
      \draw[thin, gray!50](.5,0)--(.5,1);
      \filldraw (.5,.9) node[right] {\footnotesize $b$} circle (1pt);
      \filldraw (.5,.5) node[right] {\footnotesize $c$} circle (1pt);
      \filldraw (.5,.1) node[right] {\footnotesize $d$} circle (1pt);
      \node at (-.1,.5) {\footnotesize $a$};
      \label{subfig:B_3}
    \end{tikzpicture}}
  \hspace{30pt}
  \subfigure[$B_4$]{
    \begin{tikzpicture}[scale=2]
      \draw[thin, gray!50](0,0)--(0,1);
      \draw[thick](-.7,-.25)--(-.7,.75);
      \draw[thick](.7,-.25)--(.7,.75);
      \draw[thin, gray!50](0,1)--(-.7,.75);
      \draw[thin, gray!50](0,0)--(-.7,-.25);
      \draw[thin, gray!50](0,1)--(.7,.75);
      \draw[thin, gray!50](0,0)--(.7,-.25);
      \filldraw (0, .3) node[right] {\footnotesize $a_2$} circle (1pt);
      \filldraw (0,.6) node[right] {\footnotesize $a_1$} circle (1pt);
      \node at (-.8,.25) {\footnotesize $b$}; \node at (.8,.25)
      {\footnotesize $c$};      
      \label{subfig:B_4}
    \end{tikzpicture}}
  \caption{Three excluded minors for
    $\mathcal{P}_{U_{2,4}}$.}\label{figure: excluded minors}
\end{figure}

\begin{lemma}\label{lemma: exceptions}
  Let $\rho$ be a rank-$4$ binary-natural $2$-polymatroid.  If $\rho$
  consists of three lines, then $\rho$ is $Z_3$.  If $\rho$ consists
  of two lines and two points, with no point lying on a line and no
  parallel points, then $\rho$ is $Z_{2,2}$.
\end{lemma}

\begin{proof}
  Since $S_2$ is an excluded minor for the class
  $\mathcal{P}_{U_{2,4}}$, binary-natural polymatroids have no
  parallel lines.  First assume that $\rho$ consists of three lines.
  No line can be skew to each of the others, for otherwise contracting
  it would give $S_2$.  Hence, some line must be coplanar with each of
  the other two.  The other two lines cannot be skew, for otherwise we
  would have the excluded minor $A_4$.  Hence, each pair of lines is
  coplanar, which gives $Z_3$.

  Now assume that $\rho$ consists of two points, $a$ and $b$, and two
  lines, with no point lying on a line and no parallel points.  Take
  the principal extension of $\rho$ freely adding a rank-$2$ element
  to $\{a,b\}$, and then delete $a,b$.  The result is clearly
  binary-natural and, by the first part, is $Z_3$, so $\rho$ is
  $Z_{2,2}$.
\end{proof}

The proof of Lemma \ref{lemma: exceptions} gives additional
information, which we state next and which enters into our proof of
Theorem \ref{thm: main theorem}.

\begin{lemma}\label{lemma: exceptions part 2}
  Let $\rho$ be a rank-$4$ excluded minor for $\mathcal{P}_{U_{2,4}}$.
  If $\rho$ consists of three lines, then $\rho$ is $A_4$.  If $\rho$
  consists of two lines and two points, with no point lying on a line,
  then $\rho$ is $B_4$.
\end{lemma}

\addtocounter{thm}{-2}

\begin{proof}[Proof of Theorem \ref{thm: main theorem}]
  Using the strategy outlined before Lemma \ref{lemma: compression of
    excluded minor}, we first focus on the excluded minors in
  $\mathcal{R}_{U_{2,4}}$.  Let $(E,\rho)$ be an excluded minor in
  $\mathcal{R}_{U_{2,4}}$.  By the discussion at the beginning of
  Section \ref{section:compression} and the definition of
  $\mathcal{R}_{U_{2,4}}$, each set $X_e$ contributes at least one
  element to any $U_{2,4}$-minor of $M_\rho$, so $\rho$ has the
  options below for its numbers of points and lines.
  \begin{center}
    \begin{tabular}{|c||c|c|c|c|c|c|c|c|c|}
      \hline $\#$ points & 4 & 2 & 3 & 0 & 1 & 2 & 0 & 1 & 0 \\
      \hline $\#$ lines & 0 & 1 & 1 & 2 & 2 & 2 & 3 & 3 & 4 \\ \hline
    \end{tabular}
  \end{center}
  \noindent We analyze such polymatroids below, concluding the
  analysis with \ref{four lines}.  For the second half of the proof,
  it will be useful to note that in this analysis, we find all
  excluded minors for $\mathcal{P}_{U_{2,4}}$ with these parameters,
  not just those in $\mathcal{R}_{U_{2,4}}$.  Recall that the
  \emph{nullity} of a matroid $(E,r)$ is $|E|-r(E)$.  The nullity of
  $U_{2,4}$ is $2$.  Taking a minor of a matroid cannot increase its
  nullity, so $M_\rho$ must have nullity at least $2$, so $\rho$ must
  satisfy the inequality
  \begin{equation}\label{eq:nullity}
    \sum_{e\in E}\rho(e)\geq \rho(E)+2.
  \end{equation}
  Clearly $U_{2,4}$ is the only excluded minor for
  $\mathcal{P}_{U_{2,4}}$ that is a matroid, and $L_2$ and $S_2$ are
  the only excluded minors consisting of a line and two points, and
  two lines, respectively.  It is easy to see that these are the only
  excluded minors for $\mathcal{P}_{U_{2,4}}$ with $\rho(E)=2$.  Since
  $S_2$ is an excluded minor, no other excluded minor has parallel
  lines.
  
\begin{sublemma}\label{one line three points}
  The only excluded minor for $\mathcal{P}_{U_{2,4}}$ consisting of
  one line and three points is $B_3$.
\end{sublemma}
	
Let $\rho$ have one line $e$ and three points.  Thus, $\rho(E)\leq 3$
by Inequality (\ref{eq:nullity}), so $\rho(E)= 3$.  Since $\rho$ has
no $L_2$-minor, (i) at most one point is on $e$ and (ii) the three
points are collinear (otherwise contracting one not on $e$ would give
$L_2$).  The polymatroid satisfying these conditions and having a
point on $e$ is in $\mathcal{P}_{U_{2,4}}$ since its natural matroid
is the parallel connection of two $3$-circuits.  Thus, no point is on
$e$, so $\rho$ is $B_3$.

\begin{sublemma}\label{two lines one point}
  The only excluded minor for $\mathcal{P}_{U_{2,4}}$ consisting of
  two lines and a point is $A_3$.
\end{sublemma}

Let $\rho$ have two lines, $a$ and $b$, and a point, $c$.  Inequality
(\ref{eq:nullity}) gives $\rho(E)\leq 3$, so $\rho(E)= 3$.  Exactly
one line must span $c$ since (i) if $c\not\in\cl(a)\cup\cl(b)$, then
$\rho_{/a}$ would be $S_2$, and (ii) if $c\in\cl(a)\cap\cl(b)$, then
$M_\rho$ would be the parallel connection of two $3$-circuits, which
is binary.  Thus, $\rho$ is $A_3$.

\begin{sublemma}\label{two lines two points}
  The only excluded minor for $\mathcal{P}_{U_{2,4}}$ consisting of
  two lines and two points is $B_4$.
\end{sublemma}

Let $\rho$ have two lines, $a$ and $b$, and two points, $c$ and $d$.
Inequality (\ref{eq:nullity}) gives $\rho(E)\leq 4$.  If $\rho(E)=3$,
then, since $\rho_{\del c},\rho_{\del d} \in \mathcal{P}_{U_{2,4}}$,
the proof of \ref{two lines one point} would give
$c,d\in \cl(a)\cap\cl(b)$, so $c$ and $d$ would be parallel, contrary
to $\rho$ being an excluded minor.  Thus, $\rho(E)=4$.  Assume that
some point lies on a line, say $c\in\cl(a)$.  Since $\rho$ is
connected, $d\not\in \cl(a)\cup\cl(b)$.  Now
$\rho_{/d}\in\mathcal{P}_{U_{2,4}}$, so $c$ would be on both lines of
$\rho_{/d}$, so $\rho(\{b,c,d\})=3$, but then $M_\rho$ would be the
parallel connection of a $4$-circuit and a $3$-circuit, which is
binary.  Thus, no point lies on a line, so $\rho$ is $B_4$ by Lemma
\ref{lemma: exceptions part 2}.

\begin{sublemma}\label{three lines}
  The only excluded minors for $\mathcal{P}_{U_{2,4}}$ consisting of
  three lines are $A_4$ and $S_3$.
\end{sublemma}
	
Let $\rho$ have three lines.  Inequality (\ref{eq:nullity}) gives
$\rho(E)\leq 4$.  If $\rho(E)=3$, then all pairs of lines are
coplanar, so $\rho$ is $S_3$.  If $\rho(E)=4$, then $\rho$ is $A_4$ by
Lemma \ref{lemma: exceptions part 2}.
	
\begin{sublemma}\label{three lines one point}
  The only excluded minor for $\mathcal{P}_{U_{2,4}}$ consisting of
  three lines and one point is $A_5$.
\end{sublemma}

Let $\rho$ have a point $e$ and three lines.  Inequality
(\ref{eq:nullity}) gives $\rho(E)\leq 5$.  Having $\rho(E)=3$ would
give the contradiction that $\rho_{\del e}$ is $S_3$.  If $\rho(E)=4$,
then Lemma \ref{lemma: exceptions} would give
$\rho_{\setminus e}=Z_3$, but that gives one of three contradictions:
having no line span $e$ would imply that $\rho_{/e}$ either is $S_3$
or has an $S_2$-restriction; having some but not all lines span $e$
would yield an $A_3$-restriction; having each line span $e$ would make
$M_\rho$ binary (it would be a parallel connection of three
$3$-circuits at a common base point).  Thus, $\rho(E)=5$.

First, assume that some line spans $e$.  Since $\rho$ is connected, no
two lines are coplanar, so only one line spans $e$. Also, $e$ is not
in the rank-$4$ flat spanned by the other two lines, for otherwise
$M_\rho$ would be the parallel connection of a $3$-circuit and a
$5$-circuit, which is binary.  Thus, contracting a line that does not
span $e$ would give $A_3$.  Thus, no line spans $e$.  Then $\rho_{/e}$
is binary-natural and has three lines in rank $4$, so $\rho_{/e}$ is
$Z_3$ by Lemma \ref{lemma: exceptions}.  Hence, each pair of lines
either (i) is skew and spans $e$ or (ii) is coplanar and does not span
$e$.  Since $\rho$ is connected, at most one pair of lines is
coplanar, so at least two pairs of lines span $e$.  If exactly two
pairs of lines spanned $e$, then $M_\rho$ would be binary; to see why,
note that two circuits of $M_\rho$ would have five elements (from the
flats spanned by pairs of skew lines) and intersect in three elements,
and the only other circuit would be their symmetric difference, a
$4$-element circuit (from the coplanar lines), so $M_\rho$ would be
the parallel connection of two $3$-circuits and a $4$-circuit at a
common base point, with the base point deleted.  Thus, each pair of
lines spans $e$, so $\rho$ is $A_5$, proving \ref{three lines one
  point}.

\begin{sublemma}\label{four lines}
  The only excluded minors for $\mathcal{P}_{U_{2,4}}$ consisting of
  four lines are $S_4$ and $A_6$.
\end{sublemma}
	
Let $\rho$ have four lines, so $\rho(E)\leq 6$ by Inequality
(\ref{eq:nullity}).  For $e\in E$, we have $\rho_{\del e}\ne S_3$, so
$\rho(E)>3$.  Likewise, if $\rho(E)=4$, then $\rho(E- e)=4$ for all
$e\in E$, so $\rho_{\setminus e}=Z_3$ by Lemma \ref{lemma:
  exceptions}.  Hence, the lines of $\rho$ are pairwise coplanar and
any three span $E$, so $\rho$ is $S_4$.

When $\rho(E)>4$, since $\rho$ falls under the last case to be
treated, we already know all members of $\mathcal{R}_{U_{2,4}}$ that
could be its $2$-dual, namely, those of ranks two and three.  The only
such member of $\mathcal{R}_{U_{2,4}}$ with a $2$-dual having four
lines is $U_{2,4}$, so $\rho$ is $A_6$, the $2$-dual of $U_{2,4}$.

Having identified all excluded minors in $\mathcal{R}_{U_{2,4}}$, we
next show that $S_n$, for $n\geq 3$, are the only excluded minors for
$\mathcal{P}_{U_{2,4}}$ that compress into other excluded minors.

\begin{sublemma}\label{spikelikes}
  If an excluded minor $\rho$ for $\mathcal{P}_{U_{2,4}}$ contains a
  line $e$ such that $\rho_{\downarrow e}$ is also an excluded minor,
  then, up to relabeling, $\rho=S_n$, for some $n\geq 3$, and
  $\rho_{\downarrow e}=S_{n-1}$.
\end{sublemma}

If $(E, \rho)$ and $(E\cup e, \rho')$ are polymatroids and
$\rho'_{\downarrow e}=\rho$, we call $\rho'$ a \emph{decompression} of
$\rho$.  Compressing a point is equivalent to contracting it, so we
assume that $\rho'(e)=2$.  The only excluded minor for
$\mathcal{P}_{U_{2,4}}$ with parallel lines is $S_2$, so
$\rho(f)=\rho'(f)$, for all $f\in E$, by Equation
(\ref{eq:compression}).  If $\rho'(E)<\rho'(E\cup e)$, then
$\rho=\rho'_{\setminus e}$, which contradicts $\rho'$ being an
excluded minor.  Thus, we can assume that $\rho'(E)=\rho'(E\cup e)$,
and so $\rho(E)+1=\rho'(E)$.  With that assumption, the rank of $e$ in
the $2$-dual of $\rho'$ is also $2$, so Lemma \ref{lem:commute}
implies that $\rho'$ is a decompression of $\rho$ if and only if the
$2$-dual of $\rho'$ is a decompression of the $2$-dual of $\rho$.

We first show that no excluded minor for $\mathcal{P}_{U_{2,4}}$ is a
decompression of any excluded minor
$\rho \in \mathcal{R}_{U_{2,4}}-\{S_2\}$.  By $2$-duality, it suffices
to treat $\rho\in \{U_{2,4}, L_2, A_3,B_3,B_4\}$.  In the proof of
each assertion below, $\rho'_{\downarrow e}=\rho$ where $\rho'(e)=2$.

\begin{subsublemma}
  No excluded minor for $\mathcal{P}_{U_{2,4}}$ is a decompression of
  $U_{2,4}, L_2$, or $A_3$.
\end{subsublemma}

If $\rho=U_{2,4}$, then in $\rho'$, either two points lie on $e$,
giving an $L_2$-restriction, or at least three points are not on $e$;
in the latter case, if three such points are collinear, then $B_3$ is
a restriction, otherwise contracting a point gives an $L_2$-minor.  If
$\rho=L_2$, then $\rho'$ consists of two lines and two points in rank
$3$, and by \ref{two lines two points}, there is no such excluded
minor for $\mathcal{P}_{U_{2,4}}$.  If $\rho=A_3$, then $\rho'$
consists of three lines and a point in rank $4$, and by \ref{three
  lines one point}, there is no such excluded minor for
$\mathcal{P}_{U_{2,4}}$.

\begin{subsublemma}
  No excluded minor for $\mathcal{P}_{U_{2,4}}$ is a decompression of
  $B_3$.
\end{subsublemma}

Now let $\rho=B_3$ and let $f$ be the line in $B_3$.  We show that
assuming that $\rho'$ is an excluded minor for $\mathcal{P}_{U_{2,4}}$
gives the contradiction that some proper minor of $\rho'$ is in
$\mathcal{R}_{U_{2,4}}$.  At most one point of $\rho'$ is on $e$ since
$\rho'$ has no $L_2$-restriction; none are on $f$.  If a point is on
$e$, let it be $a$; otherwise, let $a$ be any point.  Let $b$ and $c$
be the other points.  Now $\rho'$ has no parallel points, is
connected, and has rank $4$, so $\rho'_{\setminus a}$ has rank $4$,
and so $\rho'_{\setminus a}$ is $Z_{2,2}$ by Lemma \ref{lemma:
  exceptions}.  Therefore if $a$ is on $e$, then
$\rho'_{\setminus b/c}=A_3$.  If $a$ is not on $e$, then both
$\rho'_{\setminus b}$ and $\rho'_{\setminus c}$ are $Z_{2,2}$, so both
of $b$ and $c$ are in the plane spanned by $a$ and $f$, and in the
plane spanned by $a$ and $e$; thus, $a$, $b$, and $c$ are collinear,
but this gives $B_3$ as a restriction.

\begin{subsublemma}\label{sslA4B4}
  No excluded minor for $\mathcal{P}_{U_{2,4}}$ is a decompression of
  $B_4$.
\end{subsublemma}

Let $\rho=B_4$ where $f$ and $g$ are the lines and $a$ and $b$ are the
points.  Then $\rho'$ would consist of three lines and two points in
rank $5$.  Assume that $\rho'$ is an excluded minor for
$\mathcal{P}_{U_{2,4}}$.  At most one of $a$ and $b$ can be on $e$
since $\rho'$ has no $L_2$-restriction; let $a$ not be on $e$.  Now
$\rho'_{/a}$ is a rank-$4$ binary-natural polymatroid with three lines
and one point.  Also, $\rho'$ is connected, so $\rho'_{\setminus b}$
has rank $5$, so $\rho'_{/a\setminus b}$ has rank $4$, and so Lemma
\ref{lemma: exceptions} gives $\rho'_{/a\setminus b}=Z_3$.  Thus,
$\rho'_{/a}$ is $Z_3$ with the point $b$ on all lines (the other
options for $b$ yield $A_3$, $S_3$, or $S_2$ as a minor, as in the
first paragraph of \ref{three lines one point}).  Thus,
$\rho'_{/a}(\{b,f,g\})=3$, so $\rho'(\{a,b,f,g\})=4$, contrary to
having $\rho'(E)=\rho'(E\cup e)$, as noted above.

To complete the proof of Theorem \ref{thm: main theorem}, we analyze
the decompressions of $S_n$.

\begin{subsublemma}
  For $n\geq 2$, the only excluded minor for $\mathcal{P}_{U_{2,4}}$
  that is a decompression of $S_n$ is $S_{n+1}$.
\end{subsublemma}

Let $\rho'$ be an excluded minor for $\mathcal{P}_{U_{2,4}}$ having a
line $e$ with $\rho'_{\downarrow e}=\rho=S_n$.  If $n=2$, then $\rho'$
has rank $3$ and three lines, no two parallel, so $\rho'=S_3$.  Now
consider $n>2$.  As shown earlier, $\rho'(E\cup e)=\rho'(E)=n+1$.  Let
$f$ and $g$ be distinct lines in $E$, so $\rho(\{f,g\})=3$.  If
$\rho'(\{f,g\})=4$, then $\rho'(\{e,f,g\})=4$, but, by Lemma
\ref{lemma: exceptions}, this gives the contradiction that
$\rho'_{|\{e,f,g\}}\not\in \mathcal{P}_{U_{2,4}}$; thus,
$\rho'(\{f,g\})=3$.  Consider $X$ with
$\emptyset\subsetneq X\subsetneq E$, so $\rho(X)=|X|+1$.  All pairs of
lines in $X$ are coplanar in $\rho'$, so $\rho'(X)\leq |X|+1$; also,
$\rho(X)\leq \rho'(X)$, and so $\rho'(X)=|X|+1=\rho(X)$.  Thus, $X$
does not span $e$, so $\rho'(X\cup e) >\rho'(X)$.

Lastly, we show that $\rho'(X\cup e) =|X\cup e|+1$, that is,
$\rho'(X\cup e) =\rho'(X)+1$, if $\emptyset\subsetneq X\subsetneq E$.
Let $\mathcal{I}=\{ X\subseteq E\,:\, \rho'(X\cup e)=\rho'(X)+2\}$.
Note that if $X\in\mathcal{I}$ and $Y\subseteq X$, then
$Y\in\mathcal{I}$.  Since $\rho'(X\cup e) >\rho'(X)$ for
$X\subsetneq E$, to show that $\rho'=S_{n+1}$, it suffices to show
that $\mathcal{I}=\{\emptyset\}$.  Assume, instead, that
$\mathcal{I}\ne \{\emptyset\}$.  Thus, $\rho'(\{e,f\})=4$ for some
$f\in E$.  For any $g\in E-f$, we have $\rho'(\{e,g\})\ne 3$ since
$\rho'_{|\{e,f,g\}}$ cannot be $A_4$, so $\rho'(\{e,g\})=4$, and so
$\{g\}\in\mathcal{I}$.  Also, $n=\rho(E-g)=\rho'(E-g)$ and
$\rho'(E\cup e)=n+1$, so $E-g\not\in\mathcal{I}$.  Let $Z$ be a
minimal subset of $E$ not in $\mathcal{I}$.  Thus, $2\leq |Z|<n$ and
$\rho'(Z\cup e)=\rho'(Z)+1=|Z|+2$.  Then $\rho'_{/e|Z}$ has rank
$|Z|$, and for all $Y\subsetneq Z$, we have
$\rho'_{/e|Z}(Y)=\rho'(Y)$; since $\rho'(\emptyset)=0$ and
$\rho'(Y)=|Y|+1$ if $Y\ne\emptyset$, we get $\rho'_{/e|Z}= S_{|Z|}$.
This contradicts $\rho'$ being an excluded minor, so
$\mathcal{I}=\{\emptyset\}$, as needed.
\end{proof}

\addtocounter{thm}{2}

Since the set of excluded minors for $\mathcal{P}_{U_{2,4}}$ contains
an infinite family and that family is related by compression, we note
a corollary of Theorem \ref{thm: main theorem} using a slight
variation of minors.  We define a polymatroid $\rho_1$ to be a
\emph{c-minor} of a polymatroid $\rho_2$ if $\rho_1$ can be obtained
from $\rho_2$ through a series of deletions, contractions, and
compressions.

\begin{cor}
  The excluded c-minors for $\mathcal{P}_{U_{2,4}}$ are the excluded
  minors for $\mathcal{P}_{U_{2,4}}$, except for $S_n$ for $n\geq 3$.
\end{cor}

\section{Series-parallel matroids and $M(K_4)$}\label{section:K_4}

Let $\mathcal{P}_{M(K_4)}$ be the class of $2$-polymatroids whose
natural matroids have no $M(K_4)$-minor.  We will find the excluded
minors for this class.  Our motivation is series-parallel matroids
(see \cite{series-parallel,oxley}).  The excluded minors for the class
of series-parallel matroids are $U_{2,4}$ and $M(K_4)$.  Since we
already have the excluded minors for $\mathcal{P}_{U_{2,4}}$, finding
those for $\mathcal{P}_{M(K_4)}$ will help us find those for the class
of $2$-polymatroids whose natural matroids are series-parallel.

We use the strategy of the proof in Section \ref{section:main
  theorem}.  Let $\mathcal{R}_{M(K_4)}$ be the set of excluded minors
$\rho$ for $\mathcal{P}_{M(K_4)}$ such that
$\rho_{\downarrow e}\in \mathcal{P}_{M(K_4)}$ for all lines $e$ of
$\rho$.  By Lemma \ref{lemma: compression of excluded minor}, an
excluded minor is in $\mathcal{R}_{M(K_4)}$ if and only if no
compression of it by a line is an excluded minor for
$\mathcal{P}_{M(K_4)}$.  Once we find $\mathcal{R}_{M(K_4)}$, we can
find all other excluded minors for $\mathcal{P}_{M(K_4)}$ by looking
at decompressions of those in $\mathcal{R}_{M(K_4)}$.

The arguments used at the start of Section \ref{section:compression}
adapt as follows to give information about an excluded minor
$(E, \rho)$ in $\mathcal{R}_{M(K_4)}$.  When $\rho(e)=1$, the
corresponding set in $M_\rho$ is $X_e=\{e_1\}$ and $e_1$ is in each
$M(K_4)$-minor of $M_\rho$.  If $\rho(e)=2$, then $X_e=\{e_1, e_2\}$,
and to get an $M(K_4)$-minor of $M_\rho$, we must either
\begin{enumerate}
\item delete one of $e_1$ and $e_2$ and contract the other,
\item delete one of $e_1$ and $e_2$ and have the other in the
  $M(K_4)$-minor,
\item contract one of $e_1$ and $e_2$ and have the other in the
  $M(K_4)$-minor, or
\item have both $e_1$ and $e_2$ in the $M(K_4)$-minor.
\end{enumerate}
Case (1) is compression, which does not occur for an excluded minor in
$\mathcal{R}_{M(K_4)}$.  We claim that of the three other cases, only
case (3) can occur.

\begin{lemma}\label{lemma: only contraction}
  Let $(E,\rho)\in\mathcal{R}_{M(K_4)}$.  Let $A$ and $B$ be disjoint
  subsets of $E(M_\rho)$ such that $M_\rho\setminus A/B$ is isomorphic
  to $M(K_4)$.  For $e\in E$, if $\rho(e)=2$, then $|X_e\cap A|=0$ and
  $|X_e\cap B|=1$.
\end{lemma}

\begin{proof}
  Let $e\in E$ be a line, so $X_e=\{e_1, e_2\}$.  Since $e_1$ and
  $e_2$ are clones in $M_\rho$, they are clones in each minor of
  $M_\rho$ that contains both.  No two elements of $M(K_4)$ are
  clones, so case (4) cannot occur.  To finish the proof, we now rule
  out case (2).
	
  Assume that $e_1\in A$ and $e_2$ is in the $M(K_4)$-minor.  Then
  $M_\rho$ has a minor $M$ such that $M\setminus e_1$ is isomorphic to
  $M(K_4)$.  Let $L_1$ and $L_2$ be the $3$-point lines of
  $M\setminus e_1$ that contain $e_2$.  Since $e_1$ and $e_2$ are
  clones, both $L_1$ and $L_2$ span $e_1$ in $M$, so $e_1$ and $e_2$
  are parallel in $M$.  For $e_1$ and $e_2$ to be parallel in $M$ but
  not in $M_\rho$, some minor $N$ of $M_\rho$ must have an element
  $f_1$ for which $N|\{e_1,e_2,f_1\}$ is $U_{2,3}$ and $M$ is a minor
  of $N/f_1$.  Thus, $f_1\in\cl_N(\{e_1,e_2\})$.  Now $N/f_1$ has an
  $M(K_4)$-minor.  Since $\rho\in \mathcal{R}_{M(K_4)}$, the only
  points we can contract to get the $M(K_4)$-minor fall under case (3)
  above, so $f_1$ comes from the set $X_f=\{f_1,f_2\}$ for some
  $f\in E$ with $\rho(f)=2$, and $f_2$ is in the $M(K_4)$-minor.
  Since $f_1$ and $f_2$ are clones in $M_\rho$, they are clones in
  $N$, and since $f_1\in\cl_N(\{e_1,e_2\})$, we have
  $f_2\in\cl_N(\{e_1,e_2\})$.  Thus, $e_1$, $e_2$, and $f_2$ are
  parallel in $N/f_1$.  Since $e_2$ and $f_2$ are both in the
  $M(K_4)$-minor, they are parallel in that minor, but $M(K_4)$ has no
  parallel elements.  This contradiction completes the proof that case
  (2) does not arise.
\end{proof}

Thus, if $\rho\in \mathcal{R}_{M(K_4)}$ and $e\in E(\rho)$, then
exactly one point $e_1\in X_e$ is in any $M(K_4)$-minor of $M_\rho$, so
$|E(\rho)|=6$.  Hence, we can let $E=E(\rho)=E(M(K_4))$.

Let $A\subseteq E$.  Define a $2$-polymatroid $(E,\rho_A)$ by, for
$X\subseteq E$,
\begin{equation}\label{equation: rho_a}
	\rho_A(X)=r_{M(K_4)}(X)+|X\cap A|.
\end{equation}
Thus, $\rho_\emptyset=r_{M(K_4)}$.  It is easy to check that $\rho_A$
is a $2$-polymatroid and that $\rho_A\in\mathcal{R}_{M(K_4)}$; in
particular, if $A_1\subseteq X_A$ with $|A_1\cap X_e|=1$ for each
$e\in A$, then $M_{\rho_A}/A_1$ is isomorphic to $M(K_4)$. Note that
for any automorphism $\tau$ of $M(K_4)$, the excluded minors $\rho_A$
and $\rho_{\tau(A)}$ are isomorphic.  In Theorem \ref{thm:
  contraction-only cases}, we show that if
$\rho \in \mathcal{R}_{M(K_4)}$ and $A$ is its set of lines, then
$\rho=\rho_A$.  We pave the way by next characterizing the independent
sets of $M_{\rho_A}$.

\begin{lemma}\label{lemma: circuits of rho_A}
  Let $\rho_A$ be as above.  For $S\subseteq X_E$, let
  $S'= \{e\in E\,:\,|S\cap X_e|=\rho_A(e)\}$.  The set $S$ is
  independent in $M_{\rho_A}$ if and only if $S'$ is independent in
  $M(K_4)$.  Also, $S$ is a basis of $M_{\rho_A}$ if and only if $S'$
  is a basis of $M(K_4)$ and $X_e\cap S\ne \emptyset$ for each
  $e\in A$.
\end{lemma}

\begin{proof}
  We first show that if $Y\subset X_E$ and $e\in A$ (so $\rho_A(e)=2$)
  with $Y\cap X_e=\emptyset$, and if $e_i\in X_e$, then
  $r_{M_{\rho_A}}(Y)< r_{M_{\rho_A}}(Y\cup e_i)$, i.e., $e_i$ is a
  coloop of $M_{\rho_A}|Y\cup e_i$.  To see this, note that by
  Equations (\ref{eq:NatMtdRk}) and (\ref{equation: rho_a}), for any
  $Z\subseteq X_E$,
  $$r_{M_{\rho_A}}(Z)=\min\{r_{M(K_4)}(D)+|D\cap A| +|Z-X_D|\,:\,
  D\subseteq E\}.$$ Fix a set $D\subseteq E$ for which
  $r_{M_{\rho_A}}(Y\cup e_i) = r_{M(K_4)}(D)+|D\cap A| +|(Y\cup e_i)
  -X_D|$.  We get $r_{M_{\rho_A}}(Y)< r_{M_{\rho_A}}(Y\cup e_i)$ since
  if $e\not\in D$, then $|Y -X_D| <|(Y\cup e_i) -X_D|$, so
  $$r_{M_{\rho_A}}(Y)\leq r_{M(K_4)}(D)+|D\cap A| +|Y
  -X_D|<r_{M_{\rho_A}}(Y\cup e_i),$$ while if $e\in D$, then
  $|(D-e)\cap A|<|D\cap A|$, so
  $$r_{M_{\rho_A}}(Y)\leq 
  r_{M(K_4)}(D-e)+|(D-e)\cap A| +|Y -X_{D-e}|<r_{M_{\rho_A}}(Y\cup
  e_i).$$

  Turning to the sets $S$ and $S'$ in the lemma, it follows that (i)
  $S$ is independent in $M_{\rho_A}$ if and only if $X_{S'}$ is, and
  (ii) if $S$ is a basis of $M_{\rho_A}$ and $e\in A$, then
  $S\cap X_e\ne \emptyset$.  Now
  $r_{M_{\rho_A}}(X_{S'}) = \rho_A(S') = r_{M(K_4)}(S')+|S'\cap A|$,
  so $X_{S'}$ is independent in $M_{\rho_A}$ if and only if
  $r_{M(K_4)}(S')+|S'\cap A|=|X_{S'}|$, that is,
  $r_{M(K_4)}(S')=|S'|$, as needed to prove the assertion about
  independent sets.  The assertion about bases now follows.
\end{proof}

\begin{thm}\label{thm: contraction-only cases}
  If $(E,\rho)\in\mathcal{R}_{M(K_4)}$, then $\rho=\rho_A$ where
  $A = \{e\in E\,:\,\rho(e)=2\}$ and $\rho_A$ is given by Equation
  \emph{(\ref{equation: rho_a})}.
\end{thm}

\begin{proof}
  To show that $\rho=\rho_A$, it suffices to show that $M_\rho$ and
  $M_{\rho_A}$ have the same bases.
  
  Let $X_e = \{e_1,e_2\}$ if $e\in A$, and $X_e = \{e_1\}$ if
  $e\in E-A$.  Set $A_1 =\{ e_1\,:\,e\in E\}$ and
  $A_2 =\{ e_2\,:\,e\in A\}$.  By Lemma \ref{lemma: only contraction},
  there is an isomorphism $\iota: M_\rho/A_2\to M(K_4)$; label
  $M(K_4)$ so that $\iota(e_1)= e$ for all $e_1\in A_1$.

  If $A_2$ contained a circuit $C$ of $M_\rho$, and $e_2\in C$, then
  its clone $e_1$ would be in $\cl_{M_\rho}(C)$ and so would be a loop
  of $M_\rho/A_2$; however, $M(K_4)$ has no loops, so $A_2$ is
  independent in $M_\rho$.  Thus, the bases of $M_\rho/A_2$ are the
  subsets $T$ of $A_1$ for which $T\cup A_2$ is a basis of $M_\rho$.

  We next show that if $e\in A$, then $X_e$ is a cocircuit of
  $M_\rho$; thus, no basis of $M_\rho$ is disjoint from $X_e$.  Let
  $N=M_\rho/(A_2 -e_2)$, so $N$ is a rank-$4$ matroid on seven
  elements in which $e_1$ and $e_2$ are clones, and $N/e_2$ is
  $M(K_4)$.  Thus, $e_1$ and $e_2$ are clones in the dual matroid,
  $N^*$, and $N^*\setminus e_2$ is $M(K_4)$.  It follows that $X_e$ is
  a circuit of $N^*$, so $X_e$ is a cocircuit of $N$, and hence of
  $M_\rho$.

  Consider all subsets $B$ of $X_E$ for which, if $e\in A$, then
  $B\cap X_e\ne\emptyset$.  With such a set $B$, we get another such
  set $B'$ (possibly the same) by, for each $e\in A$ with
  $B\cap X_e=\{e_1\}$, replacing $e_1$ by $e_2$.  Since $e_1$ and
  $e_2$ are clones, $B$ is a basis of $M_\rho$ if and only if $B'$ is.
  Now $A_2\subseteq B'$, so $B'$ is a basis of $M_\rho$ if and only if
  $\iota(B'\cap A_1)$ is a basis of $M(K_4)$.  Note that
  $\iota(B'\cap A_1)= \{e\in E\,:\,|B\cap X_e|=\rho(e)\}$.  It now
  follows that the bases of $M_\rho$ are exactly those of
  $M_{\rho_A}$, as given in Lemma \ref{lemma: circuits of rho_A}.
\end{proof}

By Lemma \ref{lemma:R 2-duality}, since $M(K_4)$ is self-dual,
$\mathcal{R}_{M(K_4)}$ is closed under $2$-duality.  As we show next,
for $A\subseteq E$, the $2$-dual of $\rho_A$ is $\rho_{E-A}$.

\begin{lemma}
  The $2$-dual $\rho^*_A$ of the excluded minor $\rho_A$ in
  $\mathcal{R}_{M(K_4)}$ is isomorphic to $\rho_{E- A}$.
\end{lemma}

\begin{proof}
  Let $M(K_4)$ be $(E, r)$ and let its matroid dual $M^*(K_4)$ be
  $(E, r^*)$.  For $X\subseteq E$, we have $\rho_A(X)=r(X)+|X\cap A|$
  and $\rho_{E- A}(X)=r(X)+|X- A|$.  Now
  \begin{align*}
    \rho^*_A(X)&=2|X|-\rho_A(E)+\rho_A(E- X)\\
               &=2|X|-r(E)-|A|+r(E- X)+|A- X|\\
               &=|X|-r(E)+r(E- X)+|X- A|\\
               &=r^*(X)+|X-A|.
  \end{align*}
  Since $M(K_4)$ is self-dual, it follows that
  $\rho_A^*$ is isomorphic to $\rho_{E-A}$.
\end{proof}

By Theorem \ref{thm: contraction-only cases} and the remark on
isomorphism before Lemma \ref{lemma: circuits of rho_A}, there are
eleven excluded minors in $\mathcal{R}_{M(K_4)}$:
\begin{itemize}
\item there is a single excluded minor $\rho_{A}$ when $|A|$ is $0$,
  $1$, $5$, or $6$;
\item there are two when $|A|=2$, according to whether or not $A$ is a
  flat;
\item there are three when $|A|=3$, according to whether (i) $A$ is a
  circuit, (ii) $E-A$ is a circuit, or (iii) neither $A$ nor $E- A$ is
  a circuit (case (iii) gives the only self-$2$-dual excluded minor);
  and
\item two when $|A|=4$, according to whether or not $E- A$ is a flat.
\end{itemize}

In Theorem \ref{thm:cocompK4} below, we show that
$\mathcal{P}_{M(K_4)}$ has no other excluded minors.  The proof uses
the following definition and lemma.  A set $X$ in a polymatroid $\rho$
is \emph{independent} if $\rho(x)>0$ for all $x\in X$ and
$$\rho(X) = \sum_{x\in X}\rho(x).$$ Note that  $X$ is independent in
$M(K_4)$ if and only if it is independent in the excluded minor
$\rho_A$ for $\mathcal{P}_{M(K_4)}$.

\begin{lemma}\label{lem:indcomp}
  Let $\rho$ be a polymatroid on a set $E\cup g$. If $X$ is an
  independent set of $\rho_{\downarrow g}$ with
  $\rho(x)=\rho_{\downarrow g}(x)$ for all $x\in X$, then
  $\rho(X)=\rho_{\downarrow g}(X)$ and so $X$ is an independent set of
  $\rho$.
\end{lemma}

\begin{proof}
  This follows from the inequalities
  $$\rho_{\downarrow g}(X)\leq \rho(X)\leq \sum_{x\in X}\rho(x) =
  \sum_{x\in X}\rho_{\downarrow g}(x)=\rho_{\downarrow
    g}(X).\qedhere$$
\end{proof}

\begin{thm}\label{thm:cocompK4}
  All excluded minors for $\mathcal{P}_{M(K_4)}$ are in
  $\mathcal{R}_{M(K_4)}$.
\end{thm}

\begin{proof}
  Let $E$ be the ground set of $M(K_4)$.  Fix $g\not\in E$.  It
  suffices to show that there is no excluded minor $(E\cup g, \rho)$
  for $\mathcal{P}_{M(K_4)}$ where $\rho(g)=2$ and
  $\rho_{\downarrow g}=\rho_A$ for some $A\subseteq E$.  We prove this
  by contradiction; assume that $\rho_{\downarrow g}=\rho_A$.  Thus,
  for $X\subseteq E$, $g\not\in\cl_\rho(X)$ if and only if
  $\rho(X)=\rho_A(X)$.  Throughout the proof we label $M(K_4)$ as
  shown in Figure \ref{fig:MK4} and its $3$-point lines as
  $L_1=\{a,d,e\}$, $L_2=\{b,e,f\}$, $L_3=\{c,d,f\}$, and
  $L_4=\{a,b,c\}$.

  \begin{figure}[t]
    \centering
    \begin{tikzpicture}[scale=1]
      \draw[thick](0,0)--(1,1.74);%
      \draw[thick](2,0)--(1,1.74);%
      \draw[thick](0,0)--(1.5,0.87);%
      \draw[thick](2,0)--(0.5,0.87);%

      \filldraw (0,0) node[left] {\footnotesize$e$} circle (1.8pt); %
      \filldraw (2,0) node[right] {\footnotesize$c$} circle (1.8pt);%
      \filldraw (1,1.74) node[left] {\footnotesize$a$} circle (1.8pt);%
      \filldraw (0.5,0.87) node[left] {\footnotesize$d$} circle (1.8pt);%
      \filldraw (1.5,0.87) node[right] {\footnotesize$b$} circle (1.8pt);%
      \filldraw (1,0.57) node[above] {\footnotesize$f$} circle (1.8pt);%
    \end{tikzpicture}
    \caption{The labeling of $M(K_4)$ used in the proof of Theorem
      \ref{thm:cocompK4}.}\label{fig:MK4}
  \end{figure}
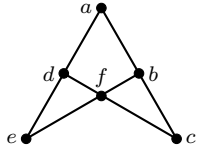

   \begin{sublemma}\label{lem:addelts}
     Let $L=\{x,y,z\}$ be a line of $M(K_4)$.  Assume that
     $\rho(L) -\rho(\{x,y\}) = |A\cap z|$.  If $Y\subset E\cup g$ with
     $Y\cap L=\{x,y\}$ and $\rho(Y)=\rho_A(Y)$, then
     $\rho(Y\cup z)=\rho_A(Y\cup z)$.
   \end{sublemma}

   \begin{proof}
     We have
     $\rho(Y\cup z)-\rho(Y)\leq \rho(L)-\rho(\{x,y\})= |A\cap z|$
     by submodularity and the hypothesis.  Since $\rho(Y)=\rho_A(Y)$,
     this gives
     $\rho(Y\cup z) \leq \rho_A(Y)+ |A\cap z| = \rho_A(Y\cup z)$.
     We also have $\rho_A(Y\cup z)\leq \rho(Y\cup z)$, so
     $\rho(Y\cup z)=\rho_A(Y\cup z)$.
   \end{proof}

   \begin{sublemma}\label{lem:E-1}
     If $x\in\cl_\rho(g)$, then $\rho(E-x)=\rho_A(E-x)$, so
     $\rho(X)=\rho_A(X)$ for all $X\subseteq E-x$.
   \end{sublemma}
   
   \begin{proof}
     By symmetry, it suffices to take $x=a$.  Consider
     $X\subseteq E-a$.  If $g\in\cl_\rho(X)$, then $a\in\cl_\rho(X)$,
     and so $\rho(X)=\rho(X\cup a)$; thus,
     $\rho_{\downarrow g}(X)=\rho_{\downarrow g}(X\cup a)$.  So
     whenever $\rho_A(X)<\rho_A(X\cup a)$, we have
     $g\not\in\cl_\rho(X)$, and so $\rho(X)=\rho_A(X)$.  Thus,
     $\rho(L_i)=\rho_A(L_i)$ for $i\in\{2,3\}$, and if $y\in E-a$,
     then $\rho(y)=\rho_A(y)$.  Lemma \ref{lem:indcomp} now gives
     $\rho(X)=\rho_A(X)$ if $X\subseteq E-a$ and $X$ is independent in
     $M(K_4)$.  Now $\rho(\{d,e,f\})=\rho_A(\{d,e,f\})$ and
     $$\rho(L_3)-\rho(\{d,f\}) =
     \rho_A(L_3)-\rho_A(\{d,f\})=|A\cap c|,$$ so item
     \ref{lem:addelts} gives $\rho(\{c,d,e,f\})=\rho_A(\{c,d,e,f\})$.
     Another application of item \ref{lem:addelts}, using $L_2$, gives
     $\rho(E-a)=\rho_A(E-a)$.  This equality implies that
     $g\not\in \cl_\rho(E-a)$, from which the second part of the
     statement follows.
   \end{proof}
     
   The next three items limit what the restriction of $\rho$ to a set
   $\{g,x\}$ with $x\in E$ can be.

   \begin{sublemma}\label{statement:lineonline}
     If $x\in E$ and $\rho(x)=2$, then $\rho(\{g,x\})>2$.
   \end{sublemma}

   \begin{proof}
     Assume instead that $\rho(g) = \rho(x)=\rho(\{x,g\})=2$.  Thus,
     $\rho_{\downarrow g}(x)=1$, so $x\not\in A$.  We claim that
     $\rho_{\del g} =\rho_{A\cup x}$, contrary to $\rho$ being an
     excluded minor.  Since $x\in\cl_\rho(g)$, item \ref{lem:E-1}
     gives $\rho(X)=\rho_A(X)=\rho_{A\cup x}(X)$ for all
     $X\subseteq E-x$.  If $X\subseteq E$ with $x\in X$, then
     $g\in \cl_\rho(X)$, so $\rho(X)=\rho_A(X)+1=\rho_{A\cup x}(X)$
     since $x\not\in A$.  So, as claimed,
     $\rho_{\del g} =\rho_{A\cup x}$.
   \end{proof}
   
   \begin{sublemma}\label{statement:pointonline}
     If $x\in E$ and $\rho(x)=1$, then $\rho(\{g,x\})>2$, that is,
     $x\not\in\cl_\rho(g)$.
   \end{sublemma}

   \begin{proof}
     Assume to the contrary
     that $\rho(x)=1$ and $x\in\cl_\rho(g)$.  Thus,
     $\rho_A(x)=\rho_{\downarrow g}(x)=1$, so $x\not\in A$.
     
     Since $x\in\cl_\rho(g)$, item \ref{lem:E-1} gives
     $\rho(E-x)=\rho_A(E-x)$.  If $x\in\cl_\rho(E-x)$, then
     $$\rho_A(E)\leq \rho(E)=\rho(E-x)=\rho_A(E-x)\leq
     \rho_A(E),$$ so $\rho_A(E)=\rho(E)$.  Therefore,
     $g\not\in\cl_\rho(E)$, and so if $X\subseteq E$, then
     $g\not\in \cl_\rho(X)$ and so $\rho(X)=\rho_A(X)$.  Thus,
     $\rho_{\del g} = \rho_A$, contrary to $\rho$ being an excluded
     minor.

     Now assume that $x\not\in\cl_\rho(E-x)$.  We claim that the
     bijection from $(E-x)\cup g$ to $E$ that fixes $E-x$ and maps $g$
     to $x$ is an isomorphism from $\rho_{\del x}$ onto
     $\rho_{A\cup x}$, contrary to $\rho$ being an excluded minor.
     Consider $X\subseteq E-x$.  Since $x\in\cl_\rho(g)$, item
     \ref{lem:E-1} gives
     $$\rho_{\del x}(X) = \rho(X)= \rho_A(X)= \rho_{A\cup x}(X).$$ Now
     $\rho(X\cup g)-\rho(X)$ is either $1$ or $2$.  First assume that
     $\rho(X\cup g)=\rho(X)+2$.  Since $x\not\in\cl_\rho(E-x)$, we
     have $\rho(X\cup x)=\rho(X)+1< \rho(X\cup g)$, so
     $g\not\in \cl_\rho(X\cup x)$.  Thus,
     $\rho(X\cup x)= \rho_A(X\cup x)$. Likewise, $\rho(X)=\rho_A(X)$.
     Thus, $\rho_A(X\cup x)=\rho_A(X)+1$.  Therefore,
     \begin{align*}
       \rho_{\del x} (X\cup g)
       & = \rho(X\cup g) \\
       & = \rho(X)+2\\
       & = \rho_A(X)+2\\
       & =  \rho_A(X\cup x)+1\\
       & = \rho_{A\cup x}(X\cup x).
     \end{align*}
     Now assume that $\rho(X\cup g)=\rho(X)+1$.  Now
     $\rho(X\cup x)=\rho(X)+ 1$ since $x\not\in\cl_\rho(E-x)$.  Thus,
     $\rho(X\cup x)=\rho(X\cup g)=\rho(X\cup \{x,g\})$ since
     $x\in\cl_\rho(g)$, and so $g\in \cl_\rho(X\cup x)$.  Thus,
     $\rho_A(X\cup x)=\rho_{\downarrow g}(X\cup x) = \rho(X\cup x)-1$.
     Now
     \begin{align*}
       \rho_{\del x} (X\cup g)
       & = \rho(X\cup g) \\
       & = \rho(X\cup x)\\
       & = \rho_A(X\cup x)+1\\
       & = \rho_{A\cup x}(X\cup x).
     \end{align*}
     Thus, $\rho_{\del x}$ is isomorphic to $\rho_{A\cup x}$.  This
     contradiction completes the proof of item
     \ref{statement:pointonline}.
   \end{proof}

   \begin{sublemma}
     There is at most one element $x\in E$ with $\rho(x)=2$ and
     $\rho(\{g,x\})=3$.
   \end{sublemma}

   \begin{proof}
     If, to the contrary, there are two different elements $x,y\in E$
     with $\rho(x)=2=\rho(y)$ and $\rho(\{g,x\})=3=\rho(\{g,y\})$,
     then
     $\rho(\{g,x,y\})\leq \rho(\{g,x\})+\rho(\{g,y\})-\rho(g) =4$, so
     $\rho_{\downarrow g}(\{x,y\})\leq 3$ while
     $\rho_{\downarrow g}(x)=2=\rho_{\downarrow g}(y)$, contrary to
     the structure of $\rho_A = \rho_{\downarrow g}$.
   \end{proof}

   Having limited what the restriction of $\rho$ to any set $\{g,x\}$,
   with $x\in E$, can be, we now turn to the broader structure of
   $\rho$.

   \begin{sublemma}
     Let $L$ and $L'$ be $3$-point lines of $M(K_4)$ with
     $\rho(L)>\rho_A(L)$ and $\rho(L')>\rho_A(L')$.  If
     $x\in L\cap L'$, then $\rho(x)=2$ and $\rho(\{g,x\})=3$.
   \end{sublemma}

   \begin{proof}
     The hypotheses give $g\in \cl_\rho(L)\cap \cl_\rho(L')$.
     Submodularity gives
     $$\rho(L\cup g)+\rho(L'\cup g)\geq \rho(L\cup L' \cup
     g)+\rho(\{g,x\}).$$ Thus,
     $\rho(L)+\rho(L')\geq \rho(L\cup L')+\rho(\{g,x\})$.
     Using $\rho(L)=\rho_A(L)+1 =3+|L\cap A|$ and the similar
     expressions for $\rho(L')$ and $\rho(L\cup L')$, we have
     $$3+|L\cap A|+3+|L'\cap A|\geq 4+|(L\cup L')\cap A|+\rho(\{g,x\}).$$ 
     Thus, $2+|A\cap x|\geq \rho(\{g,x\})$.  Now $\rho(\{g,x\})\ne 2$
     by items \ref{statement:lineonline} and
     \ref{statement:pointonline} above.  Thus, $x\in A$, so
     $\rho(x)=2$ and $\rho(\{g,x\})=3$.
   \end{proof}

   The previous two items have the next statement as an immediate
   corollary.
   
   \begin{sublemma}
     There are at most two lines among $L_1,L_2,L_3,L_4$ for which
     $\rho(L_i)>\rho_A(L_i)$.
   \end{sublemma}

   \begin{sublemma}
     There are at least two lines among $L_1,L_2,L_3,L_4$ for which
     $\rho(L_i)>\rho_A(L_i)$.
   \end{sublemma}

   \begin{proof}
     If the statement fails, then, by symmetry, we may assume that
     $\rho(L_i)=\rho_A(L_i)$ for all $i\in [3]$.  Thus, if
     $X\subseteq L_i$ for some $i\in[3]$, then
     $g \not \in\cl_\rho(X)$, and so $\rho_A(X)=\rho(X)$.  Lemma
     \ref{lem:indcomp} now gives $\rho_A(\{d,e,f\})=\rho(\{d,e,f\})$.
     Also,
     $$\rho(L_1) -\rho(\{d,e\}) = \rho_A(L_1) -\rho_A(\{d,e\}) =
     |A\cap a|.$$ Likewise, $\rho(L_2) -\rho(\{e,f\}) = |A\cap b|$ and
     $\rho(L_3) -\rho(\{d,f\}) = |A\cap c|$. With these equalities and
     $\rho(\{d,e,f\})=\rho_A(\{d,e,f\})$, we get $\rho(E)=\rho_A(E)$
     by applying item \ref{lem:addelts} three times.  Thus,
     $g\not\in\cl_\rho(E)$, so if $X\subseteq E$, then
     $g\not\in \cl_\rho(X)$ and so $\rho(X)=\rho_A(X)$.  Thus,
     $\rho_{\del g} = \rho_A$, contrary to $\rho$ being an excluded
     minor.
   \end{proof}

   In light of what is shown above and the symmetry of $M(K_4)$, we
   may now make the following two assumptions:
   \begin{itemize}
   \item[(a)] $\rho(L_1)=\rho_A(L_1)+1$ and $\rho(L_4)=\rho_A(L_4)+1$,
     and
   \item[(b)] $\rho(L_2)=\rho_A(L_2)$ and $\rho(L_3)=\rho_A(L_3)$.
   \end{itemize}
   It follows that
   \begin{itemize}
   \item[(c)] $\rho(a)=2$ and $\rho(\{a,g\})=3$, so $a\in A$,
   \item[(d)] since $g\in\cl_\rho(L_1)\cap \cl_\rho(L_4)$, if either
     $L_1\subseteq X\subseteq E$ or $L_4\subseteq X\subseteq E$, then
     $\rho(X) =\rho(X\cup g)=\rho_A(X)+1$, and
   \item[(e)] for all $x\in E-a$, we have $\rho(\{g,x\})=2+\rho(x)$,
     so $\rho(x) =\rho_A(x)$.
   \end{itemize}
   It follows that, by Lemma \ref{lem:indcomp}, if $X\subset E-a$ is
   an independent set of $M(K_4)$, then $\rho(X) =\rho_A(X)$.  With
   assumption (b), the hypotheses of item \ref{lem:addelts} apply when
   $L$ is either $L_2$ or $L_3$; since $\rho(X) =\rho_A(X)$ for each
   basis $X$ of $M(K_4)\del a$, with item \ref{lem:addelts} we get
   \begin{itemize}
   \item[(f)] $\rho(X) =\rho_A(X)$ for all subsets $X$ of $E-a$.
   \end{itemize}
   
   Since $\rho(E-a)= \rho_A(E-a)$, we get $g\not\in\cl_\rho(E-a)$.
   Thus, $\rho((E-a)\cup g)$ is either $\rho(E-a)+1$ or $\rho(E-a)+2$.

   First assume that $\rho((E-a)\cup g)=\rho(E-a)+2$.  We claim that
   this yields the contradiction that $\rho_{/g}=\rho_{A-a}$.  Fix
   $X\subseteq E$.  If $a\not\in X$, then $\rho(X\cup g)=\rho(X)+2$,
   and so $\rho_{/g}(X)=\rho(X)=\rho_A(X) =\rho_{A-a}(X)$.  Now assume
   that $a\in X$.  If $g\in \cl_\rho(X)$, then
   \begin{align*}
     \rho_{/g}(X)
     & = \rho(X\cup g)-2\\
     & = \rho(X)-2\\
     & = \rho_A(X)+1-2\\
     & = \rho_{A-a}(X).
   \end{align*}
   If $g\not\in \cl_\rho(X)$, then, since $\rho(\{a,g\})-\rho(a)=1$,
   \begin{align*}
     \rho_{/g}(X)
     & = \rho(X\cup g)-2\\
     & = \rho(X)+1-2\\
     & = \rho_A(X)-1\\
     & = \rho_{A-a}(X).
   \end{align*}
   
   Finally, assume that $\rho((E-a)\cup g)=\rho(E-a)+1$.  Since
   $g\in \cl_\rho(E)$ and $a\in A$,
   \begin{align*}
     \rho(E\cup g)
     & = \rho(E)\\
     & = \rho_A(E)+1\\
     & = \rho_A(E-a)+2\\
     & = \rho(E-a)+2.
   \end{align*}
   Thus, $\rho((E-a)\cup g)=\rho(E\cup g)-1$.
   
   We claim that the map from $(E-a)\cup g$ to $E$ that fixes each
   $x\in E-a$ and maps $g$ to $a$ is an isomorphism of $\rho_{\del a}$
   onto $\rho_A$.  By item (f) above, if $X\subseteq E-a$, then
   $\rho_{\del a}(X)=\rho_A(X)$.  We now treat the sets $X\cup g$ with
   $X$ a nonempty subset of $E-a$.  By item (e), if $x\in E-a$, then
   $\rho(\{g,x\}) = \rho(x)+2=\rho_A(\{a,x\})$.  We treat the
   remaining sets in cases below.

   \begin{sublemma}\label{lem:23,45}
     If $X\subseteq E-a$ and either $\{b,c\}\subseteq X$ or
     $\{d,e\}\subseteq X$, then $\rho_A(X\cup a)= \rho(X\cup g)$
   \end{sublemma}
  
   \begin{proof}
     Submodularity gives
     $\rho((E-a)\cup g) + \rho(X\cup \{a,g\})\geq \rho(E\cup
     g)+\rho(X\cup g)$, so, since
     $\rho((E-a)\cup g)=\rho(E\cup g)-1$, we have
     $\rho(X\cup\{a,g\})-1\geq \rho(X\cup g)$.  With item (d) above,
     this gives
     $\rho_A(X\cup a)=\rho(X\cup\{a,g\})-1 \geq \rho(X\cup
     g)>\rho_A(X)$.  Since $\rho_A(X\cup a)-\rho_A(X)=1$, we get
     $\rho_A(X\cup a)= \rho(X\cup g)$, as needed.
   \end{proof}

   The next item applies if $X$ is $L_2$, $L_3$, or a $2$-subset of
   $E-a$ other than $\{b,c\}$ and $\{d,e\}$.
   
   \begin{sublemma}
     If $X\subset E-a$ where $|X\cap L_1|\leq 1$, $|X\cap L_4|\leq 1$,
     and $r_{M(K_4)}(X)=2$, then $\rho(X\cup g) = \rho_A(X\cup a)$.
   \end{sublemma}
   
   \begin{proof}
     The hypotheses imply that, for at least one $i\in \{1,4\}$, we
     have $L_i-a=\{x,y\}$ with $x\in X$ and $y\not\in X$.  Now
     $\{x,y\}$ is either $\{b,c\}$ or $\{d,e\}$, so item
     \ref{lem:23,45} applies to supersets of $\{x,y\}$.  By
     submodularity and the cases proven above,
     \begin{align*}
       \rho(X\cup g)
       & \geq \rho(X\cup\{y,g\})+\rho(\{x,g\}) - \rho(\{x,y,g\})\\
       & = \rho_A(X\cup \{a,y\})+\rho_A(\{a,x\}) - \rho_A(L_i)\\
       & = 3+ |A\cap (X\cup \{a,y\})|+2+|A\cap\{a,x\}| - 2-|A\cap L_i|\\
       & = 3+ |A\cap(X\cup a)|\\
       & = \rho_A(X\cup a).
     \end{align*}
     Since
     $\rho_A(X\cup a) = \rho_A(X)+2= \rho(X)+2 \geq \rho(X\cup g)$,
     equality follows.
   \end{proof}

   \begin{sublemma}
     We have $\rho(\{b,d,f,g\})=\rho_A(\{a,b,d,f\})$ and
     $\rho(\{c,e,f,g\}) = \rho_A(\{a,c,e,f\})$.
   \end{sublemma}

   \begin{proof}
     The two proofs are similar; we treat the first equality.  By
     submodularity,
     $$\rho(\{b,c,d,f,g\})+ \rho(\{b,d,e,f,g\})- \rho((E-a)\cup g)
     \geq \rho(\{b,d,f,g\}),$$ so, with what we have proven,
     $\rho_A(E-e)+ \rho_A(E-c)- \rho_A(E)\geq \rho(\{b,d,f,g\})$, so
     $$3+|A-e|+ 3+ |A-c|- 3-|A|\geq \rho(\{b,d,f,g\}),$$ so
     $3+|A-\{c,e\}|\geq \rho(\{b,d,f,g\})$, that is,
     $\rho_A(\{a,b,d,f\})\geq \rho(\{b,d,f,g\})$.  Thus,
     \begin{align*}
       \rho_A(\{b,d,f\})+1
       & = \rho_A(\{a,b,d,f\})\\
       & \geq  \rho(\{b,d,f,g\})\\
       & >  \rho(\{b,d,f\})\\
       & = \rho_A(\{b,d,f\}).
     \end{align*}
     This forces $\rho_A(\{a,b,d,f\}) = \rho(\{b,d,f,g\})$, as
     desired.
   \end{proof}
 
   This establishes the final contradiction that completes the proof.
\end{proof}

Theorems \ref{thm: contraction-only cases} and \ref{thm:cocompK4} give
the following result.

\begin{thm}\label{thm:noK4exmin}
  The excluded minors for the class $\mathcal{P}_{M(K_4)}$ are the
  eleven non-isomorphic $2$-polymatroids $\rho_A$ given by Equation
  \emph{(\ref{equation: rho_a})} for subsets $A$ of $E(M(K_4))$.
\end{thm}

Having found the excluded minors for the class $\mathcal{P}_{M(K_4)}$,
we turn to the class $\mathcal{SP}$ of $2$-polymatroids whose natural
matroids are series-parallel.  We next show that the excluded minors
for $\mathcal{SP}$ are those for $\mathcal{P}_{U_{2,4}}$ together with
those for $\mathcal{P}_{M(K_4)}$.

\begin{thm}\label{thm:spexmin}
  The excluded minors for $\mathcal{SP}$ are the excluded minors for
  $\mathcal{P}_{U_{2,4}}$ together with the excluded minors for
  $\mathcal{P}_{M(K_4)}$.
\end{thm}

\begin{proof}
  Let $\rho$ be an excluded minor for $\mathcal{SP}$.  Then $\rho$
  must be an excluded minor for either $\mathcal{P}_{U_{2,4}}$ or
  $\mathcal{P}_{M(K_4)}$.  It remains to show that both sets of
  excluded minors are minor minimal for $\mathcal{SP}$.  First, we
  show that if $\rho'$ is a proper minor of an excluded minor $\rho$
  for $\mathcal{P}_{U_{2,4}}$, then $\rho'\in\mathcal{P}_{M(K_4)}$.
  If this failed, then $|E(\rho)|\geq 7$, so $\rho$ is $S_n$ for some
  $n\geq 7$.  By the proof of Proposition \ref{prop:spikelike}, each
  contraction $(S_n)_{/e}$ of $S_n$ is the uniform matroid
  $U_{n-2,n-1}$, so $(S_n)_{/e}\in\mathcal{P}_{M(K_4)}$.  Also,
  $(S_n)_{\del e}\in\mathcal{P}_{M(K_4)}$ since $S_n$ is self-$2$-dual
  and $\mathcal{P}_{M(K_4)}$ is closed under $2$-duality.  So all
  excluded minors for $\mathcal{P}_{U_{2,4}}$ are minor minimal for
  $\mathcal{SP}$.
	
  Now let $\rho_A$ be an excluded minor for $\mathcal{P}_{M(K_4)}$.
  To get the natural matroid of $\rho_A$, for each $x\in A$, take the
  $2$-sum of $M(K_4)$ with a $3$-circuit with base point $x$.  Since
  taking $2$-sums preserves $\mathbb{F}$-representability, $\rho_A$ is
  binary-natural, as needed to complete the proof.
\end{proof}

\bibliographystyle{alpha}

\end{document}